\documentclass[12pt,amstex]{amsart}

\usepackage{mathptmx}
\usepackage{mathrsfs}

\usepackage{verbatim}
\usepackage{url}
\usepackage[all]{xy}
\usepackage{color}

\usepackage[colorlinks=true,citecolor=blue]{hyperref}

\usepackage{stmaryrd}
\usepackage{epsfig}
\usepackage{amsmath}
\usepackage{amssymb}

\usepackage{amscd}
\usepackage{graphicx}
\usepackage{pstricks}

\newtheorem{thm}{Theorem}[section]
\newtheorem{lem}[thm]{Lemma}

\newtheorem{prop}[thm]{Proposition}

\newtheorem{cor}[thm]{Corollary}
\theoremstyle{definition}
\newtheorem{de}[thm]{Definition}

\theoremstyle{remark}
\newtheorem{rem}[thm]{Remark}

\numberwithin{equation}{section}

\def \N {\mathbb N}
\def \C {\mathbb C}
\def \Z {\mathbb Z}
\def \R {\mathbb R}
\def \A {\mathbb{A}}
\def \E {\mathbb{E}}

\def \B {\mathcal{B}}

\def \I {\mathcal{I}}

\def \X {\mathcal{X}}
\def \Y {\mathcal{Y}}
\def \ZZ {\mathcal{Z}}
\def \O {\mathcal{O}}

\def \RP {{\bf RP}}

\def \id {{\rm id}}
\def \h {\hat }

\def \HK {\interleave}

\def \ep {\epsilon}
\def \d {\delta}
\def \D {\Delta}

\def \lra{\longrightarrow}

\begin{document}
\title[pointwise convergence of multiple ergodic
averages]{pointwise convergence of multiple ergodic
averages and strictly ergodic models}

\author{Wen Huang}
\author{Song Shao}
\author{Xiangdong Ye}

\address{Wu Wen-Tsun Key Laboratory of Mathematics, USTC, Chinese Academy of Sciences and
School of Mathematics, University of Science and Technology of China,
Hefei, Anhui, 230026, P.R. China.}

\email{wenh@mail.ustc.edu.cn}\email{songshao@ustc.edu.cn}
\email{yexd@ustc.edu.cn}


\thanks{Authors are supported by NNSF of China (11225105, 11371339, 11431012, 11571335) and by ``the Fundamental Research Funds for the Central Universities''.}


\begin{abstract}
By building some suitable strictly ergodic models, we prove that for an
ergodic system $(X,\X,\mu, T)$, $d\in\N$, $f_1, \ldots, f_d \in
L^{\infty}(\mu)$, the averages
\begin{equation*}
    \frac{1}{N^2} \sum_{(n,m)\in [0,N-1]^2}
    f_1(T^nx)f_2(T^{n+m}x)\ldots f_d(T^{n+(d-1)m}x)
\end{equation*}
converge to a constant $\mu$ a.e.

Deriving some results from the construction, for distal systems we answer positively the question if the multiple ergodic averages converge
a.e. That is, we show that if $(X,\X,\mu, T)$ is an ergodic distal system,
and $f_1, \ldots, f_d \in L^{\infty}(\mu)$, then the multiple ergodic averages
\begin{equation*}
    \frac 1 N\sum_{n=0}^{N-1}f_1(T^nx)\ldots f_d(T^{dn}x)
\end{equation*}
converge $\mu$ a.e..

\end{abstract}

\maketitle





\section{Introduction}

\subsection{Main results}

Throughout this paper, by a {\it topological dynamical
system} (t.d.s. for short) we mean a pair $(X, T)$, where $X$ is a
compact metric space  and $T$ is a homeomorphism from $X$ to itself.
A {\em measurable system} (m.p.t. for short) is a quadruple $(X,\X,
\mu, T)$, where $(X,\X,\mu )$ is a Lebesgue probability space and $T
: X \rightarrow X$ is an invertible measure preserving
transformation.

\medskip

Let $(X,\X, \mu, T)$ be an ergodic m.p.t. We say that  $(\h{X}, \hat{T})$
is a {\em topological model} (or just a {\em model}) for $(X,\X,
\mu, T)$ if $(\h{X}, \h{T})$ is a t.d.s. and there exists an invariant
probability measure $\h{\mu}$ on the Borel $\sigma$-algebra
$\mathcal{B}(\h{X})$ such that the systems $(X,\X, \mu, T)$ and
$(\h{X}, \mathcal{B}(\h{X}), \h{\mu}, \h{T})$ are measure theoretically
isomorphic.

\medskip

The well-known Jewett-Krieger's theorem  \cite{Jewett, Krieger}
states that every ergodic system has a strictly ergodic model. We
note that one can add some additional properties to the topological
model. For example, in \cite{Lehrer} Lehrer showed that the strictly
ergodic model can be required to be a topological (strongly) mixing
system in addition.

\medskip

Now let ${\tau}_d(\h{T})=\h{T}\times \ldots\times \h{T} \ (d \ \text{times})$ and
${\sigma}_d(\h{T}) =\h{T}\times \h{T}^2\times  \ldots \times \h{T}^{d}$. The group generated by
${\tau}_d(\h{T})$ and ${\sigma}_d(\h{T})$ is denoted $\langle {\tau}_d(\h{T}), {\sigma}_d(\h{T})\rangle$.
For any $x\in \h{X}$, let $N_d(\h{X},x)=\overline{\O((x,\ldots,x),
\langle {\tau}_d(\h{T}), {\sigma}_d(\h{T})\rangle)}$, the orbit closure of
$(x,\ldots,x)$ ($d$ times) under the action of the group
$\langle {\tau}_d(\h{T}), {\sigma}_d(\h{T})\rangle$. We remark that if $(\h{X},\h{T})$ is
minimal, then all $N_d(\h{X},x)$ coincide, which will be denoted by
$N_d(\h{X})$. It was shown by Glasner \cite{G94} that if $(\h{X},\h{T})$
is minimal, then $(N_d(\h{X}), \langle {\tau}_d(\h{T}), {\sigma}_d(\h{T})\rangle)$ is
minimal.

\medskip

In this paper, first we will show the following theorem.

\medskip

\noindent {\bf Theorem A:} {\em Let $(X,\X, \mu, T)$ be an ergodic
m.p.t. and $d\in\N$. Then it has a strictly ergodic model $(\h{X}, \h{T})$ such that
$(N_d(\h{X}), \langle {\tau}_d(\h{T}), {\sigma}_d(\h{T})\rangle)$ is strictly ergodic.
}

\medskip

As a consequence, we have:

\medskip

\noindent {\bf Theorem B:} {\em Let $(X,\X,\mu, T)$ be an ergodic
m.p.t. and $d\in\N$. Then for $f_1, \ldots, f_d \in L^{\infty}(\mu)$
the averages
\begin{equation}\label{D}
    \frac{1}{N^2} \sum_{(n,m)\in [0,N-1]^2}
    f_1(T^nx)f_2(T^{n+m}x)\ldots f_d(T^{n+(d-1)m}x)
\end{equation}
converge to  a constant $\mu$ a.e.}

\medskip

We remark that similar theorems as Theorems A and B can be established for cubes (\cite{HSY-17}).
Moreover, the convergence in Theorem B can be stated for any tempered F{\rm ${\o}$}lner sequence
$\{F_N\}_{N\ge 1}$ of $\Z^2$ instead of $[0,N-1]^2$.

\medskip

It is a long open question if the multiple ergodic averages $\frac 1 N\sum_{n=0}^{N-1}f_1(T^nx)\ldots f_d(T^{dn}x)$ converge a.e.
Using some results developed when proving Theorem A, we answer the question positively for distal systems.
Namely, we have

\medskip

\noindent {\bf Theorem C:} {\em Let $(X,\X,\mu, T)$ be an ergodic distal system,
and $d\in \N$. Then for all $f_1, \ldots, f_d \in L^{\infty}(\mu)$
\begin{equation*}
    \frac 1 N\sum_{n=0}^{N-1}f_1(T^nx)\ldots f_d(T^{dn}x)
\end{equation*}
converge $\mu$ a.e.}

\medskip

Note that Furstenberg's structure theorem \cite{F} states that each ergodic system is a weakly mixing extension of an ergodic distal system. Thus, by Theorem C the open question is reduced to deal with the weakly mixing extensions.

To prove Theorem C, we show the following result, which is of independent interest.

\medskip

\noindent {\bf Theorem D:} {
Let $(X,\X,\mu,T)$ be an ergodic system and $d\in \N$. Then there exists a family $\{\mu^{(d)}_x\}_{x\in X}$ of probability measures on $X^d$ such that
\begin{enumerate}
  \item for $\mu$ a.e. $x\in X$, $\mu^{(d)}_x$ is ergodic under $T\times T^2\times \ldots \times T^d$,

  \item for all $f_1,\ldots, f_d\in L^\infty(\mu)$.
  \begin{equation}\label{ss}
\begin{split}
\frac{1}{N} \sum_{n=0}^{N-1}
   &f_1(T^nx)f_2(T^{2n}x)\ldots f_d(T^{dn}x)\\
   \longrightarrow & \int_{X^d} f_1(x_1)f_2(x_2)\ldots f_{d}(x_{d})\
   d \mu^{(d)}_x(x_1,x_2,\ldots,x_{d})
\end{split}
\end{equation}
as $N\to \infty$, where convergence is in $L^2(\mu)$.

 \item  for $\mu$ a.e. $x\in X$, $(p_j)_*(\mu_x^{(d)})\ll \mu$ for $1\le j\le d$, where $p_j: X^d\rightarrow X$ is the projection to the $j$-th coordinate.
\end{enumerate}

}

\medskip

We note that the idea which we used in this paper to show pointwise convergence theorems can be applied to other situations
(together with other tools), see for example \cite{DS-2}.
Moreover, we have the following conjecture.

\medskip
\noindent{\bf Conjecture:} {\em Let $(X,\X,\mu, T)$ be an ergodic system, then it has a topological model $(\h{X}, \h{T})$
such that for a.e. $x\in \h{X}$, $(x, \ldots,x)$ is a generic point of some ergodic measures $\mu_x^{(d)}$ invariant under
$\h{T}\times \ldots \times \h{T}^d$.}

\medskip

We also conjecture that the measures $\mu_x^{(d)}$ are the ones defined in Theorem D.
Once the conjecture is proven then the multiple ergodic averages converge a.e.
by a similar argument that we used to prove Theorem B.
\subsection{Backgrounds}

In this subsection we will give backgrounds of our research.

\subsubsection{Ergodic averages}

In this subsection we recall some
results related to pointwise ergodic averages.

The first pointwise ergodic theorem was proved by Birkhoff in 1931.
Followed from Furstenberg's beautiful work on the dynamical proof of Szemer\'{e}di's theorem in 1977, problems concerning the convergence of
multiple ergodic averages (in $L^2$ or pointwisely) attracts a lot of attention.

\medskip

The convergence of the averages
\begin{equation}\label{multiple1}
    \frac 1 N\sum_{n=0}^{N-1}f_1(T^nx)\ldots
f_d(T^{dn}x)
\end{equation}
in $L^2$ norm was established by Host and Kra \cite[Theorem 1.1]{HK05} (see also Ziegler \cite{Z}). We note that in their
proofs, the characteristic factors play a great role. The convergence of the multiple ergodic average for commuting
transformations was obtained by Tao \cite{Tao} using the finitary ergodic method, see \cite{Austin,H} for
more traditional ergodic proofs by Austin and Host respectively. Recently, the convergence of multiple ergodic averages for nilpotent group actions
was proved by Walsh \cite{Walsh}.


\medskip

The first breakthrough on pointwise convergence of (\ref{multiple1}) for $d > 1$ is due to Bourgain,
who showed in \cite{B90} that for $d = 2$, for $p,q\in \N$ and for
all $f_1, f_2 \in  L^\infty$,
the limit of $\frac{1}{N}\sum_{n=0}^{N-1} f_1(T^{pn}x)f_2(T^{qn}x)$ exists a.e. 
Before Bourgain's work, Lesigne showed this convergence holds if the
system is distal, with $T^p$, $T^q$ and $T^{p-q}$ ergodic \cite{L87}. Also in \cite{DL,Assani98}, it was shown
that the problem of the almost everywhere convergence of (\ref{multiple1}) can be deduced
to the case when the m.p.t. has zero entropy.
One can also find some results dealing with
weakly mixing transformations in \cite{Assani98}.


\medskip

Recently there are some results on the limiting behavior of the averages along cubes,
and we refer to \cite{Bergelson00,HK05,Assani,CF} for details.
Also in \cite{CF}, Chu and Frantzikinakis obtained the following result.
For $i=1,2,\ldots,d$, let $T_i: X \rightarrow X$ be m.p.t., $f_i\in L^\infty(\mu)$ be functions,
$p_i\in \mathbb{Z}[t]$ be non-constant polynomials such that $p_i-p_j$ is non-constant for $i\neq j$,
and $b: \mathbb{N} \rightarrow \mathbb{N}$ be a sequence such that $b(N)\to \infty$ and $b(N)/N^{1/h}\to 0$
as $N\to \infty$, where $h$ is the maximum degree of the polynomials $p_i$. Then the averages
$$\frac{1}{Nb(N)}\sum_{1\le m\le N, 1\le n\le b(N)} f_1(T_1^{m+p_1(n)}x)\ldots f_d (T_d^{m+p_d(n)} x) $$
converge pointwise as $N\to \infty$.


\subsubsection{Topological model}

The pioneering work on topological model was done by Jewett in
\cite{Jewett}. He proved the theorem under the additional assumption
that $T$ is weakly mixing,  and the general case was proved by Krieger in \cite{Krieger} soon.
The papers of Hansel and Raoult \cite{HanRao}, Bellow and Furstenberg \cite{BeF}
and Denker \cite{Denker}, gave different proofs of the theorem in
the general ergodic case (see also \cite{DCS}).
One can add some additional properties to the topological model. For example, in \cite{Lehrer}
Lehrer showed that the strictly ergodic model can be required as a
topological (strongly) mixing system in addition. Our Theorem A strengthens Jewett-Krieger Theorem
in other direction, i.e. we can require the model to be strictly ergodic under some group actions
on some subsets of the product space.

It is well known that each m.p.t. has a topological model
\cite{F}. There are universal models, models for some group actions and models for some special classes.
Weiss \cite{Weiss89} showed the following nice result:
there exists a minimal t.d.s. $(X, T)$ with the property that for
every aperiodic ergodic m.p.t. $(Y,\Y,\nu, S)$ there exists a
$T$-invariant Borel probability measure $\mu$ on $X$ such that the
systems $(Y,\Y,\nu, S)$ and $(X,\B(X),\mu, T)$ are measure
theoretically isomorphic.
Weiss \cite{Weiss85} (see also \cite{Weiss00, Glasner, GW06}) showed that Jewett-Krieger Theorem can be
generalized from $\Z$-actions to commutative group actions.
An ergodic system has a doubly
minimal model if and only if it has zero entropy \cite{Weiss95}
(other topological models for zero entropy systems can be found in
\cite{HM, DownL}); and an ergodic system has a strictly ergodic, UPE
(uniform positive entropy) model if and only if it has positive
entropy \cite{GW94}.

Note that not any dynamical properties can be added in the uniquely ergodic models. For example,
Lindernstrauss showed that every ergodic measure distal system $(X,\X,\mu,T )$ has a minimal
topologically distal model \cite{Lindenstrauss}.
This topological model needs not, in general, to be uniquely ergodic. In other words there
are measurable distal systems for which no uniquely ergodic topologically distal models
exist \cite{Lindenstrauss}. We refer to \cite{GW06} for more information on the topic.

\medskip

We say that $\h{\pi}: \h{X}\rightarrow \h{Y}$ is a {\em topological
model} for a factor map $\pi: (X,\X, \mu, T)\rightarrow (Y,\Y, \nu, S)$ if
$\h{\pi}$ is a topological factor map and there exist measure
theoretical isomorphisms $\phi$ and $\psi$ such that the diagram
\[
\begin{CD}
X @>{\phi}>> \h{X}\\
@V{\pi}VV      @VV{\h{\pi}}V\\
Y @>{\psi }>> \h{Y}
\end{CD}
\]
is commutative, i.e. $\h{\pi}\phi=\psi\pi$. Weiss \cite{Weiss85}
generalized Jewett-Krieger Theorem to the relative case.
Namely, he proved that if $\pi: (X,\X, \mu, T)\rightarrow (Y,\Y, \nu,
S)$ is a factor map with $(X,\X, \mu, T)$ ergodic and
$(\h{Y},\h{\Y}, \h{\nu}, \h{S})$ is a uniquely ergodic model for $(Y,\Y,
\nu, T)$, then there is a uniquely ergodic model $(\h{X}, \h{\X},
\h{\mu}, \h{T})$ for $(X,\X, \mu, T)$ and a factor map $\h{\pi}:
\h{X}\rightarrow \h{Y}$ which is a model for $\pi: X\rightarrow Y$.
We will refer this theorem as {\it Weiss's Theorem}.
We note that in \cite{Weiss85} Weiss pointed that the relative case holds
for commutative group actions.

\subsection{Main ideas of the proofs}

Now we describe the main ideas and ingredients in the proofs.

To prove Theorem A 
the first fact we face is that for an ergodic m.p.t. $(X,\X, \mu, T)$,
not every strictly ergodic model is the one we need in Theorem A
\footnote{Take any weakly mixing strictly ergodic model $(X,T,\mu)$ of an m.d.s.
with discrete spectrum, then $(N_d({X}), \langle{\tau}_d, {\sigma}_d\rangle)$ is not
strictly ergodic under $\langle{\tau}_d, {\sigma}_d\rangle$ when $d\ge 3$, in fact
in this case $N_d({X})=X^d$ and the invariant measures
$\mu\times \ldots \times \mu\not = \mu^{(d)}$, where $\mu^{(d)}$ is defined in \cite{F77}}.
This indicates that
to obtain Theorem A, Jewett-Krieger Theorem is not enough for our purpose. Fortunately, we find that
Weiss's Theorem is a right tool.

Precisely, for $d\ge 3$ let $\pi_{d-2}: X\rightarrow Z_{d-2}$ be the factor map from
$X$ to its $d-2$-step nilfactor $Z_{d-2}$. By the results of Host-Kra-Maass in \cite{HKM}, $Z_{d-2}$ may be
regarded as a topological system in the natural way. Using Weiss's Theorem
there is a uniquely ergodic model $(\h{X}, \h{\X},
\h{\mu}, \h{T})$ for $(X,\X, \mu, T)$ and a factor map $\h{\pi}_{d-2}:
\h{X}\rightarrow Z_{d-2}$ which is a model for $\pi_{d-2}: X\rightarrow
Z_{d-2}$.
\[
\begin{CD}
X @>{\phi}>> \h{X}\\
@V{\pi_{d-2}}VV      @VV{\h{\pi}_{d-2}}V\\
Z_{d-2} @>{ }>> Z_{d-2}
\end{CD}
\]

We then show  that $(\h{X},\h{T})$ is what we need. To this aim,
we need to understand well the ergodic decomposition of $d$-fold self-joinings of $X$. We first study
the $\sigma$-algebra of $\sigma_d$-invariant sets under $\sigma_d$, and show that we always can deduce
this $\sigma$-algebra to the one on its nilfactors. Then via studying nilsystems,
we get the ergodic decomposition of Furstenberg self-joinings under the action $ \sigma_d$.
This is the main tool we develop to prove Theorem A.


Once Theorem A  is proven, Theorem B will follow by an argument using some
well known theorems related to pointwise convergence for $\Z^d$ actions
(see for example \cite{Lin} by Lindenstrauss) and for uniquely ergodic systems.




Let $(X,\X,\mu,T)$ be an ergodic distal system. Then $\pi_{d-1}: X\rightarrow Z_{d-1}$ is
a distal extension. By Furstenberg's Structure Theorem, $\pi_{d-1}$ is decomposed into
isometric extensions and inverse limit. We show that the property of almost everywhere
convergence of the multiple ergodic averages (\ref{multiple1}) is preserved by these
isometric extensions, and then we conclude Theorem C. This argument is inspired by Lesigne's work in \cite{L87}.

\subsection{Organization of the paper}

We organize the paper as follows. In Section \ref{section-pre} we introduce some basic notions and
results needed in the paper. In Section \ref{section-furstenberg-joining}
we study the ergodic decomposition of self-joinings under $T\times T^2\times \ldots \times T^d$.
Then in Sections \ref{section-thA}
and \ref{section-thC}, we prove Theorems A, B, C and D respectively.

\bigskip

\noindent{\bf Acknowledgments:} 
We thank the referee for the very careful reading
and many useful comments, which help us to improve the writing of the paper and
simplify some proofs. In particular, the comments help us to rewrite Proposition \ref{Zd-erg-decom}, and obtain Corollary \ref{coro3.3} and Corollary \ref{cor4.3}
which simplify the proof of Theorem A.

\section{Preliminaries}\label{section-pre}

In this section we introduce some basic notions in ergodic theory
and topological dynamics.
In this paper, instead of just considering a single transformation $T$,
we will consider commuting
transformations $T_1$, $\ldots$ , $T_k$ of $X$. We only recall some basic
definitions and properties of systems for
one transformation. Extensions to the general case are
straightforward.

\subsection{Ergodic theory and topological dynamics}

\subsubsection{Measurable systems}



For a m.p.t. $(X,\X, \mu, T)$ we write $\I(X,\X,\mu, T)$ for the $\sigma$-algebra
$\{A\in \X : T^{-1}A = A\}$ of invariant sets. Sometimes we will use $\I$ or $\I (T)$ for short.
A m.p.t. is {\em ergodic} if all
the $T$-invariant sets have measure either $0$ or $1$. $(X,\X, \mu,
T)$ is {\em weakly mixing} if the product system $(X\times X,
\X\times \X, \mu\times \mu, T\times T)$ is erdogic.

\medskip

A {\em homomorphism} from m.p.t. $(X,\X, \mu, T)$ to $(Y,\Y, \nu,
S)$ is a measurable map $\pi : X_0 \rightarrow  Y_0$, where $X_0$ is
a $T$-invariant subset of $X$ and $Y_0$ is an $S$-invariant subset
of $Y$, both of full measure, such that $\pi_*\mu=\mu\circ
\pi^{-1}=\nu$ and $S\circ \pi(x)=\pi\circ T(x)$ for $x\in X_0$. When
we have such a homomorphism we say that $(Y,\Y, \nu, S)$
is a {\em factor} of  $(X,\X, \mu , T)$. If the factor map
$\pi: X_0\rightarrow  Y_0$ can be chosen to be bijective, then we
say that $(X,\X, \mu, T)$ and $(Y,\Y, \nu, S)$ are {\em
(measure theoretically) isomorphic} (bijective maps on Lebesgue
spaces have measurable inverses). A factor can be characterized
(modulo isomorphism) by $\pi^{-1}(\Y)$, which is a $T$-invariant
sub- $\sigma$-algebra of $\X$, and conversely any $T$-invariant
sub-$\sigma$-algebra of $\X$ defines a factor. By a classical result abuse
of terminology we denote by the same letter the $\sigma$-algebra
$\Y$ and its inverse image by $\pi$. In other words, if $(Y,\Y, \nu,
S)$ is a factor of $(X,\X, \mu, T)$, we think of $\Y$ as a
sub-$\sigma$-algebra of $\X$.

We say that $(X,\X, \mu, T)$ is an {\em inverse limit} of a sequence
of factors $(X,\X_j ,\mu, T)$ if $(\X_j)_{j\in\N}$ is an increasing
sequence of $T$-invariant sub-$\sigma$-algebras such that
$\bigvee_{j\in \N}\X_j=\X$ up to sets of measure zero.

\subsubsection{Topological dynamical systems}


A t.d.s. $(X, T)$ is {\em transitive} if there exists
some point $x\in X$ whose orbit $\O(x,T)=\{T^nx: n\in \Z\}$ is dense
in $X$ and we call such a point a {\em transitive point}. The system
is {\em minimal} if the orbit of any point is dense in $X$. This
property is equivalent to saying that X and the empty set are the
only closed invariant sets in $X$.
\medskip

A {\em factor} of a t.d.s. $(X, T)$ is another t.d.s. $(Y, S)$ such that there exists a continuous and
onto map $\phi: X \rightarrow Y$ satisfying $S\circ \phi = \phi\circ
T$. In this case, $(X,T)$ is called an {\em extension } of $(Y,S)$.
The map $\phi$ is called a {\em factor map}.


\subsubsection{$M(X)$ and $M_T(X)$}

For a t.d.s. $(X,T)$, denote by $M(X)$ the set of all
probability measure on $X$. Let $M_T(X)=\{\mu\in M(X):
T_*\mu=\mu\circ T^{-1}=\mu\}$ be the set of all $T$-invariant
Borel probability measures of $X$ and $M_T^e(X)$ be the set of ergodic elements of $M_T(X)$.
It is well known that $M_T^e(X)\neq \emptyset$.

A t.d.s. $(X,T)$ is called {\em uniquely ergodic} if
there is a unique $T$-invariant probability measure on $X$. It is
called {\em strictly ergodic} if it is uniquely ergodic and minimal.

\subsubsection{Topological distal systems}

A t.d.s. $(X,T)$ (with metric $\rho$) is
called {\em topologically distal} if $\inf_{n\in \Z} \rho (T^nx,T^nx')>0$ whenever $x,x'\in X$ are distinct.

\subsection{Conditional expectation}

If $\Y$ is a $T$-invariant sub-$\sigma$-algebra of $\X$ and $f\in
L^1(\mu)$, we write $\E(f|\Y)$, or $\E_\mu(f|\Y)$ if needed, for the
{\em conditional expectation} of $f$ with respect to $\Y$. The
conditional expectation $\E(f|\Y)$ is characterized as the unique
$\Y$-measurable function in $L^2(Y,\Y, \nu)$ such that
\begin{equation}
    \int_Y g \E(f|\Y)d\nu = \int_X  g\circ \pi f d\mu
\end{equation}
for all $g\in L^2(Y,\Y, \nu)$. We will frequently make use of the
identities $$\int \E(f|\Y) \ d\mu = \int f \ d\mu \quad
\text{and}\quad T \E(f|\Y) = \E(Tf|\Y).$$ We say that a function $f$
is {\em orthogonal} to $\Y$, and we write $f \perp \Y$, when it has a zero
conditional expectation on $\Y$. If a function $f\in L^1(\mu)$ is
measurable with respect to the factor $\Y$, we write $f \in L^1(Y,
\Y, \nu)$.

\medskip

The disintegration of $\mu$ over $\nu$, written as $\mu=\int\mu_y\ d\ \nu(y)$,
is given by a measurable map
$y \mapsto \mu_y$ from $Y$ to the space of probability measures on
$X$  such  that
\begin{equation}
    \E(f|\Y)(y)=\int_X f d\mu_y
\end{equation}
$\nu$-almost everywhere.

\subsection{Joining}

\subsubsection{Joining and conditional product measure}

The notions of joining and conditional product measure are introduced by Furstenberg in \cite{F77}.
Let $(X_i, \mu_i,T_i), i = 1,\ldots, k$, be m.p.t, and let $(Y_i,\nu_i,S_i)$ be corresponding factors, and
$\pi_i:X_i\rightarrow Y_i$ the factor maps. A measure $\nu$ on
$Y=\prod_i Y_i$ defines a {\em joining} of the measures on $Y_i$ if
it is invariant under $S_1\times \ldots \times S_k$ and maps onto
$\nu_j$ under the natural map $\prod_i Y_i\rightarrow Y_j$. When $S_1=\ldots=S_k$, we then say that
$\nu$ is a $k$-fold self-joining.

Let $\nu$ be a joining of the measures on $Y_i, i=1,\ldots, k$, and
let $\mu_i=\int \mu_{X_i, y_i}\ d\nu_i(y_i)$ represent the
disintegration of $\mu_i$ with respect to $\nu_i$. Let $\mu$ be a
measure on $X=\prod_i X_i$ defined by
\begin{equation}\label{}
    \mu=\int_Y \mu_{X_1,y_1}\times \mu_{X_2,y_2}\times \ldots \times
    \mu_{X_k,y_k}\ d\nu(y_1,y_2,\ldots,y_k).
\end{equation}
Then $\mu$ is called the {\em conditional product measure with
respect to $\nu$}.

Equivalently, $\mu$ is conditional product measure relative to $\nu$
if and only if for all $k$-tuple $f_i\in L^\infty(X_i,\mu_i),
i=1,\ldots, k$
\begin{equation}
\begin{split}
\int_X
   &f_1(x_1)f_2(x_2)\ldots f_k(x_k)\ d\mu(x_1,x_2,\ldots,x_k)\\
   = & \int_Y \E(f_1|\Y_1)(y_1)\E(f_2|\Y_2)(y_2)\ldots \E(f_{k}|\Y_k)(y_{k})\
   d\nu (y_1,y_2,\ldots,y_{k}).
\end{split}
\end{equation}

\subsubsection{Relatively independent joining}

Let $(X_1,\X_1,\mu_1,T_1), (X_2,\X_2,\mu_2,T_1)$ be two systems and let
$(Y,\Y,\nu, S)$ be a common factor with $\pi_i: X_i\rightarrow Y$ for
$i = 1, 2$ the factor maps. Let $\mu_i=\int \mu_{i,y}\ d\nu(y)$
represent the disintegration of $\mu_i$ with respect to $Y$. Let
$\mu_1\times_{\Y} \mu_2$ denote the measure defined by
$$\mu_1\times_{\Y} \mu_2(A)=\int_Y \mu_{1,y}\times \mu_{2,y}\ d\nu(y),$$
for all $A\in \X_1\times \X_2$. The system $(X_1\times X_2,
\X_1\times \X_2,\mu_1\times_Y \mu_2, T_1\times T_2)$ is called the
{\em relative product} of $X_1$ and $X_2$ with respect to $Y$ and is
denoted $X_1\times_{\Y} X_2$. $\mu_1\times_{\Y} \mu_2 $ is also
called {\em relatively independent joining} of $X_1$ and $X_2$ over
$Y$.

\subsection{HK-seminorms}\
\medskip

When $f_i$, $i\in I$, are functions on the set $X$, we define a function
$\bigotimes_{i\in I} f_i$ on $X^I$ by
\begin{equation*}
    \bigotimes_{i\in I} f_i ({\bf x})=\prod_{i\in I}
    f_i(x_i),
\end{equation*}
where ${\bf x}=(x_i)\in X^I$.

\subsubsection{}
Let $(X,\X, \mu,T)$ be an ergodic system and $k\in \N$. We define a measure $\mu^{[k]}$ on $X^{2^k}$ invariant under $T^{[k]}=T\times T\times \ldots \times T$ ($2^k$ times), by
\begin{equation*}
  \mu^{[1]}=\mu\mathop{\times}_{\I(T)}\mu=\mu\times \mu;
\end{equation*}
for $k\ge 1$,
\begin{equation*}
 \mu^{[k+1]}=\mu^{[k]}\mathop{\times}_{\I ( T^{[k]} ) } \mu^{[k]}.
\end{equation*}
Write ${\bf x}=(x_0,x_1,\ldots,x_{2^k-1})$ for a point of $X^{2^k}$, we define a seminorm $\HK \cdot \HK_k$ on $L^\infty(\mu)$ by
\begin{equation}\label{}
    \HK f\HK_k=\Big( \int_{X^{2^k}} \bigotimes_{i\in \{0,1,\ldots, 2^k-1\}}
f( {\bf x})d\mu^{[k]}({\bf x})\Big)^{1/2^k}=\Big( \int_{X^{2^k}} \prod_{i=0}^{2^k-1}
f(x_i)d\mu^{[k]}({\bf x})\Big)^{1/2^k}.
\end{equation}
That $\HK \cdot \HK_k$  is a seminorm \footnote{Here for simplicity we give the formula for real functions, and one can give the formula for complex functions similarly.} can be proved as in \cite{HK05}, and we call it
{\em Host-Kra seminorm} (HK seminorm for short).

As $X$ is assumed to be ergodic, the $\sigma$-algebra $\I^{[0]}$ is
trivial and $\mu^{[1]}=\mu \times \mu$. We therefore have
$$\HK f\HK_1=\Big(\int_{X^2}f(x_0){f(x_1)}d\mu \times\mu(x_0,x_1)\Big)^{1/2}
=\Big|\int fd\mu\Big|.$$

It is showed in \cite{HK05} that for all $f_i\in L^\infty(\mu), i\in \{0,1,\ldots,2^k-1\}$,
$$\Big|\int \bigotimes_{i\in \{0,1,\ldots, 2^k-1\}}f_i d\mu^{[k]}\Big|\le \prod_{i=0}^{2^k-1}\HK f_i\HK_k.$$
The following lemma follows immediately from the definition of the
measures and the Ergodic Theorem.

\begin{lem}\label{lemmaE1}
For every integer $k\ge 0$ and every $f\in L^\infty(\mu)$, one has
\begin{equation}\label{lem2.1}
    \HK f\HK_{k+1}=\Big(\lim_{N\to \infty} \frac{1}{N}\sum_{n=0}^{N-1}
    \HK f\cdot T^n \overline{f}\HK_k^{2^k}\Big)^{1/2^{k+1}}.
\end{equation}
\end{lem}
Note that (\ref{lem2.1}) can be considered as an alternate definition of the seminorms.

\subsubsection{}

A factor $(Z,\mathcal{Z})$ of $X$ is {\em characteristic} for averages
\begin{equation}\label{F}
    \frac 1 N\sum_{n=0}^{N-1}f_1(T^nx)\ldots
f_d(T^{dn}x)
\end{equation}
if the limiting behavior of (\ref{F}) only depends on the
conditional expectation of $f_i$ with respect to $Z$:
\begin{equation*}
    ||\lim_{N\to \infty}\frac{1}{N}\sum_{n=1}^N(T^nf_1 T^{2n}f_2 \ldots T^{dn}f_d -
    T^n\E(f_1|\mathcal{Z}) T^{2n}\E(f_2|\mathcal{Z}) \ldots T^{dn}\E(f_d|\mathcal{Z}))||_{L^2}=0
\end{equation*}
for any $f_1,\ldots,f_d\in L^{\infty}(X,\X,\mu)$. The minimal characteristic factor of (\ref{F}) always exists \cite{HK05, Z}, and it is denoted by $(Z_{d-1},\mathcal{Z}_{d-1},\mu_{d-1})$.
An important property is

\begin{prop}\label{prop2.2}\cite[Lemma 4.3]{HK05}
For a $f\in L^\infty(\mu)$, $\HK f\HK_k=0$ if and only if
$\E(f|\ZZ_{k-1})=0$.
\end{prop}

\subsection{Nilsystems}\
\medskip

Let $G$ be a group. For $g, h\in G$, we write $[g, h] =
ghg^{-1}h^{-1}$ for the commutator of $g$ and $h$ and we write
$[A,B]$ for the subgroup spanned by $\{[a, b] : a \in A, b\in B\}$.
The commutator subgroups $G_j$, $j\ge 1$, are defined inductively by
setting $G_1 = G$ and $G_{j+1} = [G_j ,G]$. Let $k \ge 1$ be an
integer. We say that $G$ is {\em $k$-step nilpotent} if $G_{k+1}$ is
the trivial subgroup.


Let $G$ be a $k$-step nilpotent Lie group and $\Gamma$ a discrete
cocompact subgroup of $G$. The compact manifold $X = G/\Gamma$ is
called a {\em $k$-step nilmanifold}. The group $G$ acts on $X$ by
left translations and we write this action as $(g, x)\mapsto gx$.
The Haar measure $\mu$ of $X$ is the unique probability measure on
$X$ invariant under this action. Let $\tau\in G$ and $T$ be the
transformation $x\mapsto \tau x$ of $X$. Then $(X, \mu, T)$ is
called a {\em $k$-step nilsystem}.

Here are some basic properties of nilsystems.

\begin{thm}\cite{P, Leibman}\label{PL}
Let $(X = G/\Gamma,\mu , T )$ be a $k$-step nilsystem with $T$ the
translation by the element $t\in G$. Then:

\begin{enumerate}
\item $(X, T )$ is uniquely ergodic if and only if $(X,\mu , T )$ is
ergodic if and only if $(X, T )$ is minimal if and only if $(X, T )$
is transitive.

\item Let $Y$ be the closed orbit of some point $x\in X$. Then $Y$ can
be given the structure of a nilmanifold, $Y = H/\Lambda$, where $H$
is a closed subgroup of $G$ containing $t$ and $\Lambda$ is a closed
cocompact subgroup of $H$.

\end{enumerate}
\end{thm}

One can generalize the above results to the action of several translations. For example,
let $X= G/\Gamma$ be a nilmanifold with Haar measure $\mu$ and let
$t_1,\ldots , t_k$ be commuting elements of $G$. If the group
spanned by the translations $t_1, \ldots , t_k$ acts ergodically on
$(X,\mu)$, then $X$ is uniquely ergodic for this group. For more details, please refer to \cite{Leibman}.

\subsection{System of order $d-1$ and topological system of order $(d-1)$}\
\medskip

In \cite{HK05}, it is showed that $(Z_{d-1}, \ZZ_{d-1}, \mu_{d-1}, T)$ has a very nice structure.

\begin{thm}\cite{HK05}
Let $(X,\X,\mu , T)$ be an ergodic system and $d\in \N$. Then the
system $(Z_{d-1}, \ZZ_{d-1}, \mu_{d-1}, T)$ is a (measure theoretic) inverse
limit of $d-1$-step nilsystems.
$(Z_{d-1}, \ZZ_{d-1}, \mu_{d-1}, T)$ is called a {\em system of order $d-1$}.
\end{thm}

One also has the topological version of this notion, i.e. the topological inverse limit of nilsystems. First recall the
definition of an inverse limit of t.d.s. If
$(X_i,T_i)_{i\in \N}$ are t.d.s. with $diam(X_i)\le 1$ and
$\pi_i: X_{i+1}\rightarrow X_i$ are factor maps, the {\em inverse
limit} of the systems is defined to be the compact subset of
$\prod_{i\in \N}X_i$ given by $\{ (x_i)_{i\in \N }: \pi_i(x_{i+1}) =
x_i, i\in \N\}$, and we denote it by
$\lim\limits_{\longleftarrow}(X_i,T_i)_{i\in\N}$. It is a compact
metric space endowed with the distance $\rho((x_{i})_{i\in\N}, (y_{i})_{i\in
\N}) = \sum_{i\in \N} 1/2^i \rho_i(x_i, y_i )$, where $\rho_{i}$ is the metric in
$X_{i}$. We note that the
maps $T_i$ induce naturally a transformation $T$ on the inverse
limit.

\begin{de}\cite{HKM}
An inverse limit of $(d-1)$-step minimal
nilsystems is called a {\em topological system of order $(d-1)$}.
\end{de}

By Theorem \ref{PL}, a topological system of order $(d-1)$ is uniquely ergodic for each $d\in \N$.

\medskip

If ${\bf n} = (n_1,\ldots, n_d)\in \Z^d$ and $\ep\in \{0,1\}^d$, we
define
$${\bf n}\cdot \ep = \sum_{i=1}^d n_i\ep_i .$$

\begin{de}
Let $(X, T)$ be a t.d.s. and let $d\in \N$. The points $x, y \in X$ are
said to be {\em regionally proximal of order $d$} if for any $\d  >
0$, there exist $x', y'\in X$ and a vector ${\bf n} = (n_1,\ldots ,
n_d)\in\Z^d$ such that $\rho (x, x') < \d, \rho (y, y') <\d$, and $$
\rho (T^{{\bf n}\cdot \ep}x', T^{{\bf n}\cdot \ep}y') < \d\
\text{for any $\ep\in \{0,1\}^d\setminus \{(0,0,\ldots,0)\}$}.$$
The set of regionally proximal pairs of
order $d$ is denoted by $\RP^{[d]}$ (or by $\RP^{[d]}(X,T)$ in case of
ambiguity), and is called {\em the regionally proximal relation of
order $d$}.
\end{de}

The above definition was introduced in \cite{HKM} by Host-Kra-Maass and it was proved that for
a minimal distal system, $\RP^{[d]}$ is an equivalence relation and $X / \RP^{[d]}$ is a
topological system of order $d$. Later it was shown
that it is an equivalence relation for any minimal systems by Shao-Ye in \cite{SY}.
We will use the following theorems in the paper.

\begin{thm}\cite[Theorem 1.2]{HKM}\label{thm-SY}
Let $(X, T)$ be a minimal topologically distal system and let $d\in \N$. Then
$(X,T)$ is a topological system of order $d$ if and only if $\RP^{[d]}=\Delta_X$.
\end{thm}

\begin{thm}\cite[Subsection 5.1]{HKM}\label{thm-HKM}
Any system of order $d$ is isomorphic in the measure theoretic sense to a topological system of order $d$.
\end{thm}

\begin{lem}\cite[Lemma A.3]{DDMSY}\label{DDMSY}
Let $(X,T)$ be a system of order $d$,
then the maximal measurable and topological factors of order $j$
coincide, where $j\leq d$.
\end{lem}

\section{ergodic decomposition of self-joinings under $T\times T^2\times \ldots \times T^d$}\label{section-furstenberg-joining}

In this section we study ergodic decomposition of self-joinings under $T\times T^2\times \ldots \times T^d$. The theorems in this section are important for our proofs, and also they have their own interest.

\subsection{Furstenberg self-joining}

Let $T: X\rightarrow X$ be a map and $d\in \N$. Set
$$\tau_d=\tau_d(T)=T\times \ldots\times T \ (d \ \text{times}),$$
$$\sigma_d=\sigma_d(T)=T\times T^2 \times \ldots \times T^{d}$$ and
$$\sigma_d'=\sigma'_d(T)=\id \times T\times \ldots \times T^{d-1}=\id\times \sigma_{d-1}.$$
Note that $\langle\tau_d, \sigma_d\rangle=\langle\tau_d,\sigma_d'\rangle$.
For any $x\in {X}$, let $N_d({X},x)=\overline{\O((x,\ldots,x),
\langle\tau_d, \sigma_d\rangle)}$, the orbit closure of
$(x,\ldots,x)$ ($d$ times) under the action of the group
$\langle\tau_d, \sigma_d\rangle$. We remark that if $({X},T)$ is
minimal, then all $N_d({X},x)$ coincide, which will be denoted by
$N_d({X})$. It was shown by Glasner \cite{G94} that if $({X},T)$
is minimal, then $(N_d({X}), \langle\tau_d, \sigma_d\rangle)$ is
minimal. Hence if $(N_d(X), \langle\tau_d, \sigma_d\rangle)$ is uniquely ergodic, then it is strictly
ergodic.

\begin{de}\label{de-Furstenberg-selfjoining}
Let $(X,T)$ be a t.d.s with $\mu\in M_T(X)$. For $d\ge 1$ let $\mu^{(d)}$ the measure on $X^d$ defined by
$$\int_{X^d} \bigotimes_{j=1}^d f_j d\mu^{(d)}=\lim_{N\rightarrow +\infty} \frac{1}{N}\int_X \prod_{j=1}^d f_j(T^{jd}x) d\mu(x)$$
for $f_i\in L^{\infty}(X,\mu)$, $1\le j\le d$, where the limits exists by \cite[Theorem 1.1]{HK05}.

We call $\mu^{(d)}$ the {\em Furstenberg
self-joining}. Clearly, it is invariant under $\tau_d$ and $\sigma_d$.
\end{de}

For a t.d.s. $(X,T)$, $\mu\in M_T(X)$and $d\in\N$, it is easy
to see that
\begin{equation*}
\frac {1}{N} \sum_{n=0}^{N-1} \sigma_d^n \mu_\D^d\longrightarrow
\mu^{(d)},\ N\to \infty, \quad \text{weak$^*$ in $M(X^d)$},
\end{equation*}
where $\mu_\D^d$ is the diagonal measure on $X^d$ as defined in
\cite{F77}, i.e. it is defined on $X^d$ as follows
\begin{equation*}
    \int_{X^d} f_1(x_1)\ldots f_d(x_d)\ d \mu_\D^d(x_1,\ldots,x_d)=
    \int_X f_1(x)\ldots f_d(x)\ d\mu(x),
\end{equation*}
where $f_1,\ldots, f_d\in C(X)$.

\subsection{The $\sigma$-algebra of invariant sets under $\sigma_d=T\times T^2\times \ldots \times T^d$}

\

In this subsection we study the $\sigma$-algebra of invariant sets under $\sigma_d=T\times T^2\times \ldots \times T^d$. We will show we always can deduce this $\sigma$-algebra to the one on its nilfactors.

For a m.p.t. $(X,\X, \mu, T)$ and $d\in \mathbb{N}$, recall that a measure $\lambda$ on $X^d$ is called {\em $d$-fold self-joining} of $X$, if it is $\tau_d$-invariant and maps onto $\mu$ under the nature $j^{th}$ coordinate projection $X^d\rightarrow X$, $1\le j\le d$. The proof of the following lemma is similar to the proof of Theorem 12.1 in \cite{HK05}.

\begin{lem}\label{AP-lem-vdc}
Let $(X,\X,\mu,T)$ be an ergodic system and $d\ge 1$ be an integer.
Suppose that $\lambda$ is a $d$-fold self-joining of $X$.
Assume that $f_1,\ldots,f_d\in L^\infty(X,\mu)$ with
$\|f_j\|_\infty\le 1$ for $j=1,\ldots,d$. Then
\begin{equation}\label{AP-VDC}
\limsup_{N\to\infty}\Big\|
\frac{1}{N}\sum_{n=0}^{N-1}f_1(T^nx_1)f_2(T^{2n}x_2)\ldots
f_d(T^{dn}x_d) \Big\|_{L^2(X^d, \lambda)}\le \min_{1\le l\le d}\{l\cdot
\interleave f_l\interleave_d \}
\end{equation}
\end{lem}

\begin{proof}
We proceed by induction. For $d=1$, by the Ergodic Theorem,
$$\|\frac{1}{N}\sum_{n=0}^{N-1}T^n f_1 \|_{L^2(\mu)}\to |\int f_1d\mu|=\HK f_1 \HK_1.$$
Let $d\ge 1$ and assume that (\ref{AP-VDC}) holds for $d$ and any $d$-fold self-joining of $X$. Let
$f_1,\ldots,f_{d+1}\in L^\infty(\mu)$ with $\|f_j\|_\infty\le 1$ for
$j=1,\ldots,d+1$. Let $\lambda$ be any $d+1$-fold self-joining of $X$. Choose $l\in \{2,3,\ldots,d+1\}$. (The case $l=1$
is similar). Write
$$\xi_n=\bigotimes_{j=1}^{d+1}T^jf_j=f_1(T^nx_1)f_2(T^{2n}x_2)\ldots
f_{d+1}(T^{(d+1)n}x_{d+1}).$$

By the van der Corput lemma \cite{Bergelson87},
\begin{equation*}
    \limsup_{N\to \infty} \big\| \frac{1}{N}\sum _{n=0}^{N-1}
    \xi_n\big\|^2_{L^2(\lambda)}
    \le \limsup_{H\to \infty}\frac{1}{H} \sum_{h=0}^{H-1}\limsup_{N\to\infty}
    \left|\frac{1}{N}\sum_{n=0}^{N-1} \int \overline{\xi}_{n+h}\cdot\xi_n d\lambda \right|.
\end{equation*}
Letting $M$ denote the last $\limsup$, we need to show that $M\le
l^2\HK f_l\HK^2_{d+1}$. For any $h\ge 1$,
\begin{equation*}
\begin{split}
    & \ \ \ \left|\frac{1}{N}\sum_{n=0}^{N-1} \int \overline{\xi}_{n+h}\cdot\xi_n d\lambda
    \right| \\ &=\left | \int(f_1\cdot T^h\overline{f}_1)\otimes \frac{1}{N}\sum_{n=0}^{N-1}
    (\sigma_d)^n \bigotimes_{j=2}^{d+1}f_j\cdot T^{jh}\overline{f}_j d
    \lambda(x_1,\ldots,x_{d+1})\right|\\
    &\le \Big\|f_1\cdot T^h\overline{f}_1 \Big\|_{L^2(\lambda)}\cdot
    \Big\|\frac{1}{N}\sum_{n=0}^{N-1}
    (\sigma_d)^n \bigotimes_{j=2}^{d+1}f_j\cdot T^{jh}\overline{f}_j
    \Big\|_{L^2(\lambda)}\\
    &= \Big\|f_1\cdot T^h\overline{f}_1 \Big\|_{L^2(\mu)}\cdot
    \Big\|\frac{1}{N}\sum_{n=0}^{N-1}
    (\sigma_d)^n \bigotimes_{j=2}^{d+1}f_j\cdot T^{jh}\overline{f}_j
    \Big\|_{L^2(\lambda')}
\end{split}
\end{equation*}
where $\lambda'$ is the image of $\lambda$ to the last $d$ coordinates.
It is clear $\lambda'$ is a $d$-fold self-joining of $X$, and by the inductive assumption,
\begin{equation*}
   \left|\frac{1}{N}\sum_{n=0}^{N-1} \int \overline{\xi}_{n+h}\cdot\xi_n d\lambda
    \right|\le l\HK f_l\cdot T^{lh}\overline{f}_l\HK_d.
\end{equation*}
We get
\begin{equation*}
\begin{split}
    M& \le l\cdot \limsup_{H\to\infty}
    \frac{1}{H}\sum_{h=0}^{H-1}\HK f_l\cdot T^{lh}\overline{f}_l\HK_d
    \le l^2\cdot\limsup_{H\to\infty}\frac{1}{H}\sum_{h=0}^{H-1} \HK f_l\cdot T^h
    \overline{f}_l\HK_d\\
    &\le l^2\cdot \limsup_{H\to\infty} \Big( \frac{1}{H}\sum_{h=0}^{H-1} \HK f_l\cdot T^h
    \overline{f}_l\HK_d^{2^d} \Big)^{1/2^d}\\
    &=l^2\cdot \HK f_l\HK_{d+1}^2.
\end{split}
\end{equation*}
The last equation follows from Lemma \ref{lemmaE1}. The proof is completed.
\end{proof}

\begin{cor}\label{coro3.3}
Let $(X,\X,\mu,T)$ be an ergodic system and $d\ge 2$ be an integer.
Suppose that $\lambda$ is a $d$-fold self-joining of $X$ and it is invariant under $\sigma_d$.
Assume that $f_1,\ldots,f_d\in L^\infty(X,\mu)$ with
$\|f_j\|_\infty\le 1$ for $j=1,\ldots,d$. Then
\begin{equation}
\Big |
\int f_1(x_1)f_2(x_2)\ldots
f_d(x_d) d \lambda(x_1,\ldots, x_d)\Big | \le d \min_{1\le l\le d}\{
\interleave f_l\interleave_{d-1} \}
\end{equation}
\end{cor}

\begin{lem}\label{AP-lem-vdc2}
Let $(X,\X,\mu,T)$ be an ergodic system and $d\in \N$. Suppose that $\lambda$ is a $d$-fold self-joining of $X$ and it is $\sigma_d$-invariant. Assume that
$f_1,\ldots,f_d\in L^\infty(X,\mu)$. Then
\begin{equation}\label{}
    \E\Big(\bigotimes_{j=1}^d f_j\Big|\I(X^d, \X^d, \lambda,\sigma_d)\Big)  =\E\Big( \bigotimes_{j=1}^d
    \E(f_j|\ZZ_{d-1})\Big |\I(X^d, \X^d, \lambda,\sigma_d) \Big).
\end{equation}

\end{lem}

\begin{proof}
By telescoping, it suffices to show that
\begin{equation}\label{}
\E\Big(\bigotimes_{j=1}^d f_j\Big|\I(X^d,\X^d, \lambda,\sigma_d)\Big)
=0
\end{equation}
whenever $\E(f_k|\ZZ_{d-1})=0$ for some $k\in \{1,2,\ldots,d\}$.
This condition implies that $\HK f_k\HK_d=0$ by Proposition \ref{prop2.2}. By the Ergodic Theorem
and Lemma \ref{AP-lem-vdc}, we have
\begin{equation*}
\begin{split}
&\ \ \ \ \Big\|\E\Big(\bigotimes_{j=1}^d f_j\Big|\I(X^d, \X^d,
\lambda,\sigma_d)\Big)\Big\|_{L^2(\lambda)} \\
&=\lim_{N\to\infty}\Big\|
\frac{1}{N}\sum_{n=0}^{N-1}f_1(T^nx_1)f_2(T^{2n}x_2)\ldots
f_d(T^{dn}x_d) \Big\|_{L^2(\lambda)}\le k \cdot\interleave
f_k\interleave_d =0.
\end{split}
\end{equation*}
So the lemma follows.
\end{proof}

\begin{prop}
Let $(X,\X,\mu,T)$ be ergodic and $d\in \N$. Suppose that
$\lambda$ is a $d$-fold self-joining of $X$ and
it is $\sigma_d$-invariant. Then the
$\sigma$-algebra $\I(X^d, \X^d, \lambda,\sigma_d)$ is measurable with
respect to $\ZZ_{d-1}^{d}$.
\end{prop}

\begin{proof}
Every bounded function on $X^{d}$ which is measurable with respect
to \linebreak $\I(X^d, \lambda,\sigma_d)$ can be approximated in
$L^2(\lambda)$ by finite sums of functions of the form
$\E(\bigotimes_{j=1}^d f_j | \I( X^d, \X^d, \lambda, \sigma_d))$ where
$f_1,\ldots,f_d$ are bounded functions on $X$. By Lemma
\ref{AP-lem-vdc2}, one can assume that these functions are
measurable with respect to $Z_{d-1}$. In this case
$\bigotimes_{j=1}^df_j$ is measurable with respect to
$\ZZ_{d-1}^{d}$. Since this $\sigma$-algebra $\ZZ_{d-1}^{d}$ is
invariant under $\sigma_d$, $\E(\bigotimes_{j=1}^df_j|\I(X^d, \X^d,
\lambda,\sigma_d))=\lim \limits_{N\rightarrow +\infty}\frac{1}{N}\sum \limits_{n=0}^{N-1}
\big(\bigotimes_{j=1}^df_j\big)\circ \sigma_d^n$ is also measurable with respect to
$\ZZ_{d-1}^{d}$. Therefore $\I(X^d, \X^d, \lambda,\sigma_d)$ is
measurable with respect to $\ZZ_{d-1}^{d}$.
\end{proof}

Let $\pi: (X,\X,\mu,T)\rightarrow (Y,\Y,\nu,S)$ be a homomorphism. $\pi$ is {\em  ergodic}
or $(X,\X,\mu,T)$ is an {\em ergodic extension} of $(Y,\Y,\nu,S)$ if $T-$invariant sets of $\X$ is contained in $\Y$, i.e. $\I(T)\subseteq \Y$.

\begin{cor}\label{AP-cor-ergodic}
Let $(X, \X, \mu,T)$ be an ergodic system and $d\in \N$. Suppose that $\lambda$
is a $d$-fold self-joining of $X$ and it is $\sigma_d$-invariant. Then the
factor map $\pi_{d-1}^d: (X^{d}, \X^d, \lambda, \sigma_d)\rightarrow
(Z_{d-1}^d, \ZZ^d_{d-1}, \widetilde{\lambda},\sigma_d)$ is ergodic, where $\widetilde{\lambda}$ is the image of $\lambda$.

In particular, one has that $\I(X^{d}, \lambda, \sigma_d)$ is isomorphic to $\I(Z_{d-1}^d, \widetilde{\lambda},\sigma_d)$.
\end{cor}


\subsection{Ergodic decomposition of Furstenberg self-joining of nilsystems under action \texorpdfstring{$\sigma_d$}{sigma d}}
\

In the previous subsection we show that to study the $\sigma$-algebra of invariant sets
under $\sigma_d=T\times T^2\times \ldots \times T^d$, we only need to study the one on its nilfactors.
Hence in this subsection we study the ergodic decomposition of Furstenberg self-joinings of nilsystems under
the action $ \sigma_d$.


In this subsection $d\ge 2$ is an integer, and $(X =
Z_{d-1}, \ZZ_{d-1}, \mu_{d-1}, T )$ is a topological system of order $d-1$. Recall
$$N_\ell=N_\ell(X)=\overline{\O(\Delta_\ell({X}),
\sigma_\ell)}=\overline{\O((x,\ldots,x), \langle\tau_\ell, \sigma_\ell\rangle)}\subset
X^\ell$$ and $$N_\ell[x]:=\overline{\O((x,\ldots,x), \sigma'_\ell)}=\{x\}\times \overline{\O((x,\ldots,x), \sigma_{\ell-1})},$$ where
$\ell\ge 2$ and $x\in X$.

\medskip

\subsubsection{Basic properties}

First we recall some basic properties.

\begin{thm}\cite{BHK05, Z05}\label{ziegler}
With the notations above, we have
\begin{enumerate}
  \item (Theorem \ref{PL}) The $(N_d, \langle\tau_d,\sigma_d\rangle)$ is ergodic (and thus uniquely
  ergodic) with the Furstenberg self-joining $\mu^{(d)}_{d-1}$.
  \item (Theorem \ref{PL}) For each $x\in X$, the system $(N_d[x],\sigma_d')$ is
uniquely ergodic with some measure $\delta_x\times \mu^{(d)}_{d-1,x}$.
  \item \cite[Lemma 5.3]{BHK05} $\displaystyle \mu^{(d)}_{d-1} = \int_X \d_x\times \mu^{(d)}_{d-1,x}\
  d\mu_{d-1}(x)$.
  \item (Ziegler) Let $f_1, f_2, \ldots , f_{d-1}$ be continuous functions on
$X$ and let $\{M_i\}$ and $\{N_i\}$ be two sequences of integers
such that $N_i\to \infty$. For $\mu_{d-1}$-almost every $x\in X$,
\begin{equation}
\begin{split}
\frac{1}{N_i} \sum_{n=M_i}^{N_i+M_i-1}
   &f_1(T^nx)f_2(T^{2n}x)\ldots f_{d-1}(T^{(d-1)n}x)\\
   \rightarrow & \int f_1(x_1)f_2(x_2)\ldots f_{d-1}(x_{d-1})\
   d\mu^{(d)}_{d-1,x}(x_1,x_2,\ldots,x_{d-1})
\end{split}
\end{equation}
as $i\to \infty$.
\end{enumerate}
\end{thm}

\begin{rem}
In fact, in \cite{BHK05, Z05} Theorem \ref{ziegler} is for nilsystems. Via an inverse limit argument,
it is easy to see that Theorem \ref{ziegler} holds for topological systems of order $d$.
\end{rem}

\subsubsection{The ergodic decomposition of $\mu^{(d)}_{d-1}$ under $\sigma_d$}

Now we study the ergodic decomposition of $\mu^{(d)}_{d-1}$ under
$\sigma_d$. By Theorem \ref{PL}, for each $x\in X$, let $\nu^{(d)}_{d-1,x}$ be the unique
$\sigma_d$-invariant measure on $\overline{\O(x^d,\sigma_d)}$, where
$x^d=(x,x,\ldots,x)\in X^d$. Then
$$\varphi: X \longrightarrow
M(N_d);\ \ \ x\mapsto \nu^{(d)}_{d-1,x}$$ is a Borel map and $\varphi(X)\subseteq M_{\sigma_d}^{e}(N_d)$. This
fact follows from that $x \mapsto
\frac{1}{N}\sum_{n<N}\d_{\sigma_d^n x^d}$ is continuous and
$\frac{1}{N}\sum_{n<N}\d_{\sigma_d^nx^d}$ converges to
$\nu^{(d)}_{d-1,x}$ weakly.

It is easy to check that $\displaystyle \int_{X}
\nu^{(d)}_{d-1,x} \ d \mu_{d-1}(x)$ is $\langle\tau_d,\sigma_d\rangle$-invariant and
hence it is equal to $\mu^{(d)}_{d-1}$ by the uniqueness. Hence we have
\begin{equation}\label{h9}
    \mu^{(d)}_{d-1}=\int_{X} \nu^{(d)}_{d-1,x} \ d \mu_{d-1}(x).
\end{equation}

\medskip

Now we will prove the following result:

\begin{thm}\label{de-Zd}
$\mu^{(d)}_{d-1}=\int_{X} \nu^{(d)}_{d-1,x} \ d \mu_{d-1}(x)$ is the ergodic decomposition of
$\mu^{(d)}_{d-1}$ under $\sigma_d$.
\end{thm}

First we have the following claim:

\medskip

\noindent {\bf Claim.} \ {\em There exists a continuous map $\psi: N_d\rightarrow X$ such that
$$x_1=\psi(x_2,x_3,\cdots,x_{d+1})$$ for every $(x_1,x_2,x_3,\cdots,x_{d+1})\in N_{d+1}$.}

\medskip

\noindent {\em Proof of Claim.}
The claim follows from the following fact:
the projection $$ p_2: N_{d+1}(X)\rightarrow N_d(X);\ \  (x_1,x_2,\cdots,x_{d+1})\mapsto (x_2,\cdots,x_{d+1})$$
is a bijection.

By definition, it is clear that $p_2$ is onto. Now we show that $p_2$ is also injective. Let $(x_1,x_2,\cdots,x_{d+1}), (y_1,x_2,\cdots,x_{d+1})\in N_{d+1}$. We will show that $x_1=y_1$.
First by definition of $N_{d+1}(X)$,  there exists $(x^*,y^*)\in \overline{\O((x_1,y_1),T\times T)}$ and $x\in X$ such that
$$ (x^*,x,\cdots,x), (y^*,x,\cdots,x)\in N_{d+1}.$$
Thus for  for any $\d>0$, there is some $n,m\in \Z$ such that $\rho(T^{m+jn}x, x)<\d/2$ for all $j=1,2,\ldots, d$ and $\rho(T^{m}x, x^*)<\d/2$.
Let $x'=T^m x, y'=T^{m+n}x$, and let ${\bf n}=(n,n,\ldots,n)\in \mathbb{Z}^{d-1}$. Then $\rho(x',x^*)<\d/2$, $\rho(y',x)<\d/2$ and
$$\{{\bf n}\cdot \ep: \ep\in \{0,1\}^{d-1}\setminus \{(0,0,\ldots,0)\} \}=\{n,2n,\ldots,(d-1)n\}.$$
So we have that
\begin{equation*}
  \begin{split}
  & \rho(T^{{\bf n}\cdot \ep}x',T^{{\bf n}\cdot \ep}y') = \rho(T^{{\bf n}\cdot \ep}T^m x,T^{{\bf n}\cdot \ep}
T^{n+m}x)\\ &  \le \rho(T^{{\bf n}\cdot \ep}T^mx,x)+\rho(x,T^{{\bf n}\cdot \ep}
T^{n+m}x)\\ & \le 2 \max_{1\le j\le d} \rho(T^{jn+m}x, x) = \d .
  \end{split}
\end{equation*}
By the definition of $\RP^{[d-1]}(X)$, one has that $(x^*,x)\in \RP^{[d-1]}(X)$. Hence $x^*=x $ by Theorem \ref{thm-SY}.
Similarly, one has that $y^*=x$. Thus $x^*=y^*$ and so $x_1=y_1$ since $(X,T)$ is distal. This shows that $p_2$ is injective. The proof of Claim is completed. \hfill $\square$


\medskip


Now we show that $\varphi: X \longrightarrow
M(N_d);\  x\mapsto \nu^{(d)}_{d-1,x}$ is one-to-one.
Since $\psi^{-1}(x)\supset \overline{\O((x^d,\sigma_d)}$ for any $x\in X$, one has that
$$\overline{\O((x^d,\sigma_d)}\cap \overline{\O((y^d,\sigma_d)}=\emptyset$$
whenever $x\neq y$. Thus $\nu_{d-1,x}^{(d)}\neq \nu_{d-1,y}^{(d)}$ whenever $x\neq y$.

By the above discussion, one has that $\varphi$ is a one-to-one Borel map. Hence by Souslin Theorem (See e.g.  \cite[Theorem 2.8 (2)]{Glasner}),
$\varphi(X)$ is a Borel subset of $M(N_d)$ and $\varphi$ is a Borel isomorphism from $X$ to $\varphi(X)$.
Let $\kappa=\varphi_*(\mu_{d-1})$. Then $\kappa$ is a Borel probability measure on the $G_\delta$ subset
$M_{\sigma_d}^{e}(N_d)$ of $M(N_d)$, $\kappa(\varphi(X))=1$ and
$$\mu_{d-1}^{(d)}\stackrel{\eqref{h9}}{=}\int_{X} \varphi(x) \ d \mu_{d-1}(x)=\int_{M_{\sigma_d}^{e}(N_d)} \theta d \kappa(\theta)$$
is the ergodic decomposition of $\mu_{d-1}^{(d)}$ under $\sigma_d$.

Let $\mathcal{N}_d$ be  the Borel $\sigma$-algebra  of $N_d$.
By the ergodic decomposition Theorem (see e.g. \cite[Theorem 4.2]{V}), there exists a Borel map $\xi: N_d\rightarrow M_{\sigma_d}^{e}(N_d)$
such that

(i) $\xi(\sigma_d x)=\xi(x)$ for any $x\in N_d$,

(ii) for any $\theta\in M_{\sigma_d}^{e}(N_d)$, $\theta(\xi^{-1}(\theta))=1$,

(iii) for any $\eta\in M_{\sigma_d}(N_d)$,
\begin{equation}\label{eeeee11}
\eta(A)=\int_{N_d} \xi(x)(A)d\eta(x)
\end{equation}
for any $A\in \mathcal{N}_d$.

Let $\mathcal{M}_{\sigma_d}^{e}(N_d)$ be the Borel $\sigma$-algebra  of $M_{\sigma_d}^{e}(N_d)$ and
$\nu=\xi_*(\mu_{d-1}^{(d)})$. Then
$$\xi: (N_d,\mathcal{N}_d, \mu_{d-1}^{(d)})\rightarrow (M_{\sigma_d}^{e}(N_d),
\mathcal{M}_{\sigma_d}^{e}(N_d),\nu)$$
is a measure-preserving map. By \cite[Lemma 4.2]{V},
$\xi^{-1}(\mathcal{M}_{\sigma_d}^{e}(N_d))=\I(N_d,\mathcal{N}_d , \mu_{d-1}^{(d)},\sigma_d)$ (mod  $\mu_{d-1}^{(d)})$
and  $$\mu_{d-1}^{(d)}=\int_{N_d} \xi(x)\ d\mu_{d-1}^{(d)}(x)=\int_{M_{\sigma_d}^{e}(N_d)} \theta d \nu(\theta)$$ is
the disintegration of $\mu_{d-1}^{(d)}$ over $\nu$ by $\xi$.

Hence the uniqueness of the representation in Choquet's theorem implies $\nu=\kappa$ so that $\nu(\varphi(X))=1$.
Now $$\varphi^{-1}:(M_{\sigma_d}^{e}(N_d),\mathcal{M}_{\sigma_d}^{e}(N_d),\nu)\rightarrow (X,\mathcal{X},\mu_{d-1})$$ is an isomorphism.

\medskip

Let $E=\{ (x_2,x_3,\cdots,x_{d+1})\in N_d: \psi(x_2,x_3,\cdots,x_{d+1})=\phi^{-1}\circ \xi(x_2,x_3,\cdots,x_{d+1})\}$. Then
$E$ is a Borel susbet of $N_d$. Now for any $x\in X$, it is not  hard to see that
$$\overline{\O((x^d,\sigma_d)}\cap \xi^{-1}(\phi(x))\subset E$$
and so
$$\phi(x)(E) \ge \phi(x)\big(\overline{\O((x^d,\sigma_d)}\cap \xi^{-1}(\phi(x))\big)=\phi(x)\big( \xi^{-1}(\phi(x))\big)\stackrel{\text{(ii)}}{=}1.$$
Moreover
$$\mu_{d-1}^{(d)}\big(E \big)\stackrel{\eqref{h9}}{=}\int_{X} \varphi(x)\big( E \big) \ d \mu_{d-1}(x)=1.$$
This implies that $\psi=\varphi^{-1}\circ \xi$ for $\mu_{d-1}^{(d)}$-a.e.
\medskip

To sum up, we have

\begin{prop}\label{Zd-erg-decom}
Let $\psi: N_d\rightarrow X$ be the continuous map such that
$x_1=\psi(x_2,x_3,\cdots,x_{d+1})$ for every $(x_1,x_2,x_3,\cdots,x_{d+1})\in N_{d+1}$. Then $\psi: (N_d,\mathcal{N}_d, \mu_{d-1}^{(d)})
\rightarrow (X,\mathcal{X},\mu_{d-1})$ is a measure preserving map such that
\begin{enumerate}
\item $\psi^{-1}(\mathcal{X})=\I(N_d,\mathcal{N}_d , \mu_{d-1}^{(d)},\sigma_d)$ (mod  $\mu_{d-1}^{(d)})$.

\item  the disintegration $\mu^{(d)}_{d-1}= \int_X (\mu_{d-1}^{(d)})_x d\mu_{d-1}(x)$ of $\mu_{d-1}^{(d)}$ over $\mu_{d-1}$ by $\psi$ is the
ergodic decomposition of $\mu^{(d)}_{d-1}$ under $\sigma_d$.

\item  $(\mu_{d-1}^{(d)})_x=\nu^{(d)}_{d-1,x}$ for $\mu_{d-1}$-a.e $x\in X$.
\end{enumerate}
\end{prop}

\begin{rem}
Note that $N_{d+1}[x]=\{x\}\times \overline{\O(x^d,\sigma_d)}$. It follows that for all $x$
$$\mu^{(d+1)}_{d-1,x}=\nu^{(d)}_{d-1,x}.$$ Hence
$$(\mu_{d-1}^{(d)})_x=\nu^{(d)}_{d-1,x}=\mu^{(d+1)}_{d-1,x}$$ for $\mu_{d-1}$-a.e $x\in X$.
Thus usually, $(\mu_{d-1}^{(d)})_x$ is different from $\mu^{(d)}_{d-1,x}$.
\end{rem}

It is easy to see that Theorem \ref{de-Zd} follows from Proposition \ref{Zd-erg-decom}.

\subsection{Ergodic decomposition of Furstenberg self-joining under the action $ \sigma_d$}
\

Let $(X,T)$ be a minimal t.d.s with measure $\mu$, and let $\pi_{d-1}: (X,T)\rightarrow (Z_{d-1}, T)$
be the topological factor map, where $Z_{d-1}$ is both a topological system of order $d-1$
and a system of order $d-1$ with measure $\mu_{d-1}$. Notice that in the next section we will show that for each ergodic system,
one always can find such a minimal topological model.

\medskip

Recall that for a t.d.s. $(X,T)$ and $d\in\N$, we define $\mu^{(d)}$ as  the weak$^*$ limit points
of sequence $\{\frac {1}{N} \sum_{n=0}^{N-1} \sigma_d^n \mu_\D^d\}$ in $M(X^d)$.
Then $\mu^{(d)}$ is a $d$-fold self-joining of $X$ and it is $\langle\tau_d,\sigma_d\rangle$-invariant.
By definition, the image of $\mu^{(d)}$ under $\pi_{d-1}^d$ is $\mu_{d-1}^{(d)}$.

\medskip

By Corollary \ref{AP-cor-ergodic}, the factor map $\pi_{d-1}^d:
(X^{d}, \X^d, \mu^{(d)},\sigma_d)\rightarrow (Z_{d-1}^d,\ZZ_{d-1}^d
\mu_{d-1}^{(d)},\sigma_d)$ is ergodic. Hence $\I(X^d,\X^d,
\mu^{(d)},\sigma_d)=\I(Z_{d-1}^d, \ZZ_{d-1}^d, \mu_{d-1}^{(d)},\sigma_d)$. By
(\ref{h9}),
$$\displaystyle \mu^{(d)}_{d-1} = \int_{Z_{d-1}} \nu^{(d)}_{d-1,x}\ d \mu_{d-1}(x) $$
is the ergodic decomposition of $\mu^{(d)}_{d-1}$ under $\sigma_d$.

Let $\phi=\pi_{d-1}^d|_{N_d(X)}$ and $$\psi: (N_d(Z_{d-1}),\mathcal{N}_d(Z_{d-1}),
\mu_{d-1}^{(d)})\rightarrow (Z_{d-1},\mathcal{Z}_{d-1},\mu_{d-1})$$ be the
measure-preserving map defined in Proposition \ref{Zd-erg-decom}. Then by  Corollary \ref{AP-cor-ergodic} and the fact $\mu^{(d)}\big(N_d(X)\big)=1$ and $\mu_{d-1}^{(d)}\big(N_d(Z_{d-1})\big)=1$, one has $$\phi: (N_d(X), \mathcal{N}_d(X),\mu^{(d)},\sigma_d)\rightarrow (N_d(Z_{d-1}),\mathcal{N}_d(Z_{d-1}),\mu_{d-1}^{(d)},\sigma_d) $$ is
a factor map with
\begin{equation}\label{ddbb-d}
\I(N_d(X), \mathcal{N}_d(X), \mu^{(d)},\sigma_d)= \phi^{-1}\I(N_d(Z_{d-1}),\mathcal{N}_d(Z_{d-1}), \mu_{d-1}^{(d)},\sigma_d).
\end{equation}
 Combining this with Proposition \ref{Zd-erg-decom} (1), we have
\begin{align}\label{h4}
    (N_d(X), \mathcal{N}_d(X),\mu^{(d)})&\stackrel{\phi}{\lra} (N_d(Z_{d-1}),\mathcal{N}_d(Z_{d-1}), \mu_{d-1}^{(d)})\stackrel{\psi}{\lra}
    (Z_{d-1},\ZZ_{d-1},\mu_{d-1})\nonumber\\
    {\bf x}&\lra \phi({\bf x}) \lra s=\psi(\phi({\bf x}))
\end{align}
and $\phi^{-1}(\psi^{-1}(\ZZ_{d-1}))=\I(N_d(X),\mu^{(d)},\sigma_d)$.
From this, let
\begin{equation}\label{1111}
    \mu^{(d)}=\int_{Z_{d-1}}\nu^{(d)}_s\ d \mu_{d-1}(s)
\end{equation}
be the disintegration of $\mu^{(d)}$
over $\mu_{d-1}$ by  $\psi\circ \phi$. Since $$\phi^{-1}(\psi^{-1}(\ZZ_{d-1}))=\I(N_d(X),\mu^{(d)},\sigma_d),$$
\eqref{1111} is the ergodic decompositions of $\mu^{(d)}$ under $\sigma_d$ (see e.g. \cite[Theorem 8.7]{Glasner}).  Moreover, for $\mu_{d-1}$-a.e. $s\in Z_{d-1}$,
$$\phi_*(\nu^{(d)}_s)=(\mu_{d-1}^{(d)})_s=\nu^{(d)}_{d-1,s}$$  by \cite[Corollary 5.24]{EW}, \eqref{ddbb-d}   and Proposition \ref{Zd-erg-decom}, where  $\mu^{(d)}_{d-1}= \int_{Z_{d-1}} (\mu_{d-1}^{(d)})_s d\mu_{d-1}(s)$ is the disintegration of $\mu_{d-1}^{(d)}$ over $\mu_{d-1}$ by $\psi$.

\medskip

To sum up, we have the following result:

\begin{thm}\label{ergodic-deco-Fur}
$\mu^{(d)}=\int_{Z_{d-1}}\nu^{(d)}_s\ d \mu_{d-1}(s)$ is the ergodic decomposition of
$\mu^{(d)}$ under $\sigma_d$.
\end{thm}

\section{Proof of Theorem A and Theorem B}\label{section-thA}

In this section we show Theorem A and Theorem B. First  we give the  proof of Theorem A by using the tools developed in Section \ref{section-furstenberg-joining}.

\subsection{Another form of Theorem A}

\begin{de}
Let $(X,\X, \mu, T)$ be an ergodic m.p.t. and $(\h{X}, \h{T})$
be its model.
For $d\in \N$, $(\h{X}, \h{T})$ is called a
{\em $\langle\tau_d, \sigma_d\rangle-$strictly ergodic model} for $(X,\X, \mu,
T)$ if $(\h{X}, \h{T})$ is a strictly ergodic model and $(N_d(\h{X}),
\langle {\tau}_d(\h{T}), {\sigma}_d(\h{T})\rangle)$ is strictly ergodic.
\end{de}

From the statement of Theorem A one does not know  how the model looks like. The following statement
avoids this weakness and is suitable for the induction.
Note that we let $Z_0=\{pt\}$ be the trivial system.

\begin{thm}\label{model}
Let $(X,\X, \mu, T)$ be an ergodic  m.p.t, $d\ge 2$  and $\pi_{d-2}: X\lra Z_{d-2}$ be the
factor map to the $Z_{d-2}$. Assume that $Z_{d-2}$ is isomorphic to a topological system
of order $d-2$ (see Theorem \ref{thm-HKM}, still denote it by $Z_{d-2}$). Then any strictly ergodic system $\h{X}$
obtained from Weiss's theorem is a $\langle\tau_d, \sigma_d\rangle-$strictly ergodic model.

\[
\begin{CD}
X @>{}>> \h{X}\\
@V{\pi_{d-2}}VV      @VV{\h{\pi}_{d-2}}V\\
Z_{d-2} @>{ }>> Z_{d-2}
\end{CD}
\]
\end{thm}

Theorem \ref{model} follows immediately from the following corollary, which follows from Corollary \ref{coro3.3}.

\begin{cor}\label{cor4.3}
Let $(X, \X, \mu,T)$ be an ergodic system and $d\ge 2$ be an integer. Let $\pi_{d-2}: (X, \X, \mu,T)\rightarrow (Z_{d-2},\ZZ_{d-2}, \mu_{d-2}, T)$ be its factor of order $d-2$. Suppose that $\lambda$
is a $d$-fold self-joining of $X$ and it is $\sigma_d$-invariant. If $\mu_{d-2}^{(d)}$ is the image of $\lambda$ under $\pi_{d-2}^d$, then $\lambda$ is the conditionally independent measure with respect to $\mu_{d-2}^{(d)}$.
\end{cor}

The following result is a direct consequence of Theorem \ref{model}.

\begin{cor} Let $(X,\ T)$ be a uniquely  ergodic t.d.s. with invariant measure $\mu$ and $d\ge 2$. Assume that the measure theoretic factor map $\pi_{d-2}: X\lra Z_{d-2}$ is (equal $\mu$-a.e. to) a continuous factor map.
Then $(N_d(X),\langle\tau_d, \sigma_d\rangle)$ is unique ergodic, and the unique invariant measure is the Furstenberg
self-joining $\mu^{(d)}$.
\end{cor}

\subsection{Another proof of Theorem \ref{model} and its consequence}\

\medskip

In this subsection, we give another proof of Theorem \ref{model}. By this proof, we will get the following result, which is key to the proof of Theorem D.

\begin{thm}\label{thm4.5}
Let $(X,\ T)$ be a uniquely  ergodic t.d.s. with invariant measure $\mu$ and $d\ge 1$. Assume that the measure theoretic factor map $\pi_{d-1}: X\lra Z_{d-1}$ is (equal $\mu$-a.e. to) a continuous factor map. Let $\mu^{(d)}=\int_{Z_{d-1}}\nu^{(d)}_s\ d \mu_{d-1}(s)$ be the ergodic decompositions of $\mu^{(d)}$ under $\sigma_d$ as in Theorem \ref{ergodic-deco-Fur} and let
$\mu=\int _{Z_{d-1}}\theta_s\ d\mu_{d-1}(s)$ be the disintegration
of $\mu$ over $\mu_{d-1}$. Then
\begin{equation}\label{}
\mu^{(d+1)}=\int_{Z_{d-1}}\theta_s\times \nu^{(d)}_s \ d\mu_{d-1}(s).
\end{equation}
\end{thm}

\noindent {\em Another proof of Theorem \ref{model}} \
Let $(X,T)$ be a strictly ergodic
system and let $\mu$ be its unique $T$-invariant measure.

When $d=2$. Note that $X^2=X\times X$, $\tau_2=T\times T$, $\sigma_2=T\times
T^2$ and $\sigma_2'=\id \times T$. It is easy to see that $N_2(X)=X\times X$ and  $\mu\times \mu$ is unique $\langle\tau_2,\sigma_2\rangle$-invariant measure on $X\times X$.

Next assume that Theorem \ref{model} holds
for $d\ge 2$. We show it also holds for $d+1$.
Let $\pi_{d-1}: X\rightarrow Z_{d-1}$ be the factor map from $X$ to
$Z_{d-1}$, the system of order $d-1$. We build $\h{X}$ in the following way
by Weiss's theorem and Theorem \ref{thm-HKM}.
\[
\begin{CD}
X @>{}>> \h{X}\\
@V{\pi_{d-1}}VV      @VV{\h{\pi}_{d-1}}V\\
Z_{d-1} @>{ }>> Z_{d-1}\\
\end{CD}
\]
Without loss of generality we assume that $X=\h{X}$. 
Now we show that
$ (N_{d+1}(X), \allowbreak \langle \tau_{d+1}, \sigma_{d+1} \rangle)$ is uniquely ergodic.

Let $\zeta: Z_{d-1}\lra Z_{d-2}$ be the factor map to the maximal topological factor of order $d-2$.
By Theorem \ref{DDMSY}, $\zeta$ is also the factor map to the maximal factor of order $d-2$.
By the inductive assumption, $(N_{d}(X),\langle\tau_{d},\sigma_{d}\rangle)$ is uniquely ergodic,
and we denote its unique measure by $\mu^{(d)}$.

\medskip

By Theorem \ref{ergodic-deco-Fur},
\begin{equation}
    \mu^{(d)}=\int_{Z_{d-1}}\nu^{(d)}_s\ d \mu_{d-1}(s)
\end{equation}
is the ergodic decompositions of $\mu^{(d)}$ under $\sigma_d$.

\medskip

Let $\lambda$ be a
$\langle\tau_{d+1},\sigma_{d+1}\rangle$-invariant measure of $N_{d+1}(X)$ and
$\mu=\int _{Z_{d-1}}\theta_s\ d\mu_{d-1}(s)$ be the disintegration
of $\mu$ over $\mu_{d-1}$. We will show that
\begin{equation}\label{limit-D}
\lambda=\int_{Z_{d-1}}\theta_s\times \nu^{(d)}_s \ d\mu_{d-1}(s)
\end{equation}
which implies that $\lambda$ is unique.

\medskip

To do this let
$$p_1: (N_{d+1}(X), \langle\tau_{d+1},\sigma_{d+1}\rangle)\rightarrow (X, T); \
(x_1,{\bf x})\mapsto x_1$$
$$p_2: (N_{d+1}(X), \langle\tau_{d+1},\sigma_{d+1}\rangle)\rightarrow
(N_d(X), \langle\tau_d, \sigma_d\rangle); \ (x_1,{\bf x})\mapsto
{\bf x}$$  be the projections (here we use the fact that $\langle
\tau_d, T^2\times \ldots \times T^{d+1}\rangle=\langle\tau_d,
\sigma_d\rangle$). Then $(p_2)_*(\lambda)$ is a
$\langle\tau_d,\sigma_d\rangle$-invariant measure of $N_d(X)$. By
the assumption on $d$, $(p_2)_*(\lambda)=\mu^{(d)}$. Hence we may
assume that
\begin{equation}\label{h5}
\lambda =\int_{X^d} \lambda_{{\bf x}}\times \d_{\bf x}\ d \mu^{(d)} ({\bf
x})
\end{equation}
is the disintegration of $\lambda$ over $\mu^{(d)}$. Since $\lambda$ is
$\sigma'_{d+1}=\id\times \sigma_d$-invariant, we have
\begin{eqnarray*}
    \lambda &= &\id \times \sigma_d \lambda=
    \int_{X^d} \lambda_{\bf x}\times \sigma_d\d_{\bf x} \ d \mu^{(d)}({\bf x})\
    \\ & = & \int_{X^d} \lambda_{\bf x}\times \d_{\sigma_d({\bf x})} \ d \mu^{(d)}({\bf x})\\
    \\&=& \int_{X^d} \lambda_{(\sigma_d)^{-1}({\bf x})}\times \d_{\bf x} \ d \mu^{(d)}({\bf x}).
\end{eqnarray*}
The uniqueness of disintegration implies that
\begin{equation}\label{h6}
    \lambda_{(\sigma_d)^{-1}({\bf x})}=\lambda_{\bf x},\quad \mu^{(d)} \ a.e.
\end{equation}

Define $$F: (X^d,\mu^{(d)}, \sigma_d) \lra M(X): \ {{\bf x}}\mapsto
\lambda_{{\bf x}}.$$ By (\ref{h6}), $F$ is a $\sigma_d$-invariant
$M(X)$-valued function. Hence $F$ is $\I(X^d,\X^d,
\mu^{(d)},\sigma_d)$-measurable, and this implies $\lambda_{\bf
x}=\lambda_{\psi(\phi({\bf x}))}=\lambda_s, \ \mu^{(d)}\ $ a.e.,
where $s, \psi$ and $\phi$ are defined in (\ref{h4}).

Thus by (\ref{h5}) one has that
\begin{equation}\label{add-final}
\begin{split}
    \lambda& =\int_{X^d}  \lambda_{{\bf x}}\times \d_{\bf x}\ d\mu^{(d)} ({\bf x})=
\int_{X^d} \lambda_{\psi(\phi({\bf x}))} \times\d_{\bf x}\ d\mu^{(d)} ({\bf x})\\
       &=\int_{Z_{d-1}}\int_{X^d}\lambda_s\times \d_{\bf x}\ d\nu^{(d)}_s({\bf x})
       d\mu_{d-1}(s)\\
       &= \int_{Z_{d-1}}\lambda_s\times \Big(\int_{X^d} \d_{\bf x}\
       d\nu^{(d)}_s({\bf x})\Big ) d \mu_{d-1}(s)\\
       &=\int_{Z_{d-1}}\lambda_s\times \nu^{(d)}_s \ d\mu_{d-1}(s).
\end{split}
\end{equation}

In the sequel we will show that $\lambda_s=\theta_s$ for $\mu_{d-1}$-a.e. $s\in Z_{d-1}$
and it is clear that (\ref{limit-D}) follows from this fact and (\ref{add-final}) immediately.

\medskip

Let $\pi_{d-1}^{d+1}: (N_{d+1}(X), \langle\tau_{d+1},\sigma_{d+1}\rangle)\lra
(N_{d+1}(Z_{d-1}), \langle\tau_{d+1},\sigma_{d+1}\rangle)$ be the natural factor
map. By Theorem \ref{ziegler}, $(N_{d+1}(Z_{d-1}),
\langle\tau_{d+1},\sigma_{d+1}\rangle,\mu^{(d+1)}_{d-1})$ is uniquely ergodic.
Hence
\begin{equation*}
\begin{split}
    \int_{Z_{d-1}}(\pi^{d+1}_{d-1})_*(\lambda_s\times \nu^{(d)}_s) \ d\mu_{d-1}(s)=(\pi^{d+1}_{d-1})_*(\lambda) =\mu_{d-1}^{(d+1)}=\int_{Z_{d-1}}\d_{s}\times \mu^{(d+1)}_{d-1,s} \
    d\mu_{d-1}(s).
\end{split}
\end{equation*}
The last equality follows from  Theorem \ref{ziegler}(3), since
for $\mu_{d-1}$-a.e. $s\in Z_{d-1}$, the system $(\overline{\O((s,\ldots,s), \sigma'_{d+1})},\sigma_{d+1}')$ is
uniquely ergodic with some measure $\delta_s\times \mu^{(d+1)}_{d-1,s}$.
Hence $\mu^{(d+1)}_{d-1,s}$
is the unique ergodic measure of $(\overline{\O((s,\ldots,s), \sigma_{d})},\sigma_{d})$,
i.e.  $\mu^{(d+1)}_{d-1,s}=\nu^{(d)}_{d-1,s}$.

Note that $$(\pi^d_{d-1})_*(\nu^{(d)}_s)=\phi_*(\nu^{(d)}_s)=(\mu_{d-1}^{(d)})_s=\nu^{(d)}_{d-1,s}=\mu^{(d+1)}_{d-1,s}$$ and
$(\mu_{d-1}^{(d)})_s(\psi^{-1}(s))=1$
for $\mu_{d-1}$-a.e. $s\in Z_{d-1}$.
We claim that 
\begin{equation}\label{meas-two-eq}
(\pi^{d+1}_{d-1})_*(\lambda_s\times \nu^{(d)}_s)=\d_{s}\times \mu^{(d+1)}_{d-1,s}
\end{equation}
for $\mu_{d-1}$-a.e. $s\in Z_{d-1}$. We postpone the verification of (\ref{meas-two-eq})
to the next subsection.

It is clear that (\ref{meas-two-eq}) implies
\begin{equation}\label{proj=s}
(\pi_{d-1})_*(\lambda_s)=\d_s
\end{equation}
for $\mu_{d-1}$-a.e. $s\in Z_{d-1}$. Since $(p_1)_*(\lambda)=\mu$,
it follows from (\ref{add-final}) that
\begin{equation}\label{===}
\mu=\int_{Z_{d-1}}\lambda_s\ d\mu_{d-1}(s).
\end{equation}

(\ref{proj=s}) and (\ref{===}) imply that
$\mu=\int_{Z_{d-1}}\lambda_s\ d\mu_{d-1}(s)$ is also the
disintegration of $\mu$ over $\mu_{d-1}$. So we conclude
$$\lambda_s=\theta_s\  \text{for}\ \mu_{d-1}-a.e.\ s\in Z_{d-1}$$
by the uniqueness of the disintegration. The proof is completed.

\subsubsection{Proof of \eqref{meas-two-eq}}
Assume the contrary that \eqref{meas-two-eq} does not hold.
Then $$\mu_{d-1}(\{s\in Z_{d-1}: (\pi^{d+1}_{d-1})_*(\lambda_s\times \nu^{(d)}_s)\not=\d_{s}\times \mu^{(d+1)}_{d-1,s}\})>0.$$ So
there is some function $f\in C(N_{d+1}(Z_{d-1}))$ such that $\mu_{d-1}(C)>0$, where
\begin{equation*}
   C=\{s\in Z_{d-1}: (\pi^{d+1}_{d-1})_*(\lambda_s\times \nu^{(d)}_s)(f)>\d_{s}\times \mu^{(d+1)}_{d-1,s}(f)\}.
\end{equation*}
Let $B=\psi^{-1}(C)$ and  $A=p_{d-1,2}^{-1}(B)$, where $$p_{d-1,2}:
(N_{d+1}(Z_{d-1}), \langle\tau_{d+1},\sigma_{d+1}\rangle)\rightarrow
(N_d(Z_{d-1}), \langle\tau_d, \sigma_d\rangle); \ (s_1,{\bf
s})\mapsto {\bf s}$$  be the projection. Since
$\big((\pi^d_{d-1})_*(\nu^{(d)}_s)=\mu^{(d+1)}_{d-1,s}=(\mu_{d-1}^{(d)})_s$
and $(\mu_{d-1}^{(d)})_s(\psi^{-1}(s))=1$ for $\mu_{d-1}$-a.e. $s\in
Z_{d-1}$, one has that for $\mu_{d-1}$-a.e. $s\in Z_{d-1}$,
$$\mu^{(d+1)}_{d-1,s}(B)=\begin{cases} 1 &\text { if }s\in C\\ 0 &\text{ if } s\not \in C\end{cases} \text{ and }
(\pi^d_{d-1})_*(\nu^{(d)}_s)(B)=\begin{cases} 1 &\text { if }s\in C\\ 0 &\text{ if } s\not \in C.
\end{cases}$$

Moreover, for $\mu_{d-1}$-a.e. $s\in Z_{d-1}$, one has that for
$\mu_{d-1}$-a.e. $s\in Z_{d-1}$,
$$(\pi^{d+1}_{d-1})_*(\lambda_s\times \nu^{(d)}_s)(A)=(\pi^d_{d-1})_*(\nu^{(d)}_s)(B)=\begin{cases} 1 &\text { if }s\in C\\ 0 &\text{ if } s\not \in C\end{cases}$$
and
$$\d_{s}\times \mu^{(d+1)}_{d-1,s}(A)=\mu^{(d+1)}_{d-1,s}(B)=\begin{cases} 1 &\text { if }s\in C\\ 0 &\text{ if } s\not \in C\end{cases}.$$

Thus
\begin{equation*}
\begin{split}
    \mu^{(d+1)}_{d-1}(f\cdot 1_A)&=\int_{N_{d+1}(Z_{d-1})}f\cdot 1_A \ d\mu^{(d+1)}_{d-1}\\
    &=\int_{Z_{d-1}}\big(\int_{N_{d+1}(Z_{d-1})} f\cdot 1_A \ d \d_{s}\times \mu^{(d+1)}_{d-1,s}\big) \
    d\mu_{d-1}(s)\\
    &=\int_{C}\d_{s}\times \mu^{(d+1)}_{d-1,s}(f)d\mu_{d-1}(s)\\
    &<\int_{C}(\pi^{d+1}_{d-1})_*(\lambda_s\times \nu^{(d)}_s)(f)\ d\mu_{d-1}(s)\\
    &=\int_{Z_{d-1}} \big( \int_{N_{d+1}(Z_{d-1})} f\cdot 1_A \ d (\pi^{d+1}_{d-1})_*(\lambda_s \times \nu^{(d)}_s) \big) \ d\mu_{d-1}(s)\\
    &= \mu^{(d+1)}_{d-1}(f\cdot 1_A),
\end{split}
\end{equation*}
a contradiction! Hence \eqref{meas-two-eq} holds.

\subsection{Proof of Theorem B}\label{section-thB}
In this subsection we show how to obtain Theorem B from Theorem A. We need the following formula which is easy to be verified.

\begin{lem}\label{lem-product}
Let $\{a_i\}, \{b_i\}\subseteq \C$. Then
\begin{equation*}\label{}
\prod_{i=1}^k a_i-\prod_{i=1}^k b_i=(a_1-b_1)b_2\ldots b_k+
a_1(a_2-b_2)b_3\ldots b_k +\ldots +a_1\ldots a_{k-1}(a_k-b_k).
\end{equation*}
\end{lem}

\medskip

\noindent{\bf The proof of Theorem B:}
Since $(X,\X, \mu, T)$ has a $\langle \tau_d,\sigma_d\rangle-$strictly ergodic model, we
may assume that $(X,T)$ itself is a minimal t.d.s. and
$\mu$ is its unique measure such that $(N_d(X), \langle\tau_d, \sigma_d
\rangle)$ is uniquely ergodic with the unique measure $\mu^{(d)}$.

Fix $f_1,\ldots, f_d\in L^\infty$ and let $\ep>0$. Without loss of generality, we assume that for all $1\le j\le d$,
$\|f_j\|_\infty\le 1$. Choose continuous functions $g_j$ such that
$\|g_j\|_\infty\le 1$ and $\|f_j-g_j\|_1<\ep/d$ for all $1\le j\le
d$. We have
\begin{equation}\label{yexd1}
\begin{split}
&  \left | \frac{1}{N^2} \sum_{n\in [0,N-1] \atop{m\in [0,N-1]}} \prod_{j=1}^{d} f_j(T^{n+(j-1)m}x) -
\int_{N_d(X)}\bigotimes_{j=1}^df_jd\mu^{(d)} \right |\\
& \le  \left | \frac{1}{N^2} \sum_{n\in [0,N-1] \atop{m\in [0,N-1]}} \prod_{j=1}^{d} f_j(T^{n+(j-1)m}x) -
\frac{1}{N^2} \sum_{n\in [0,N-1] \atop{m\in [0,N-1]}}\prod_{j=1}^d g_j(T^{n+(j-1)m}x)\right |\\
&+\left | \frac{1}{N^2} \sum_{n\in [0,N-1] \atop{m\in [0,N-1]}}\prod_{j=1}^d g_j(T^{n+(j-1)m}x)-
\int_{N_d(X)}\bigotimes_{j=1}^dg_jd\mu^{(d)} \right |+\left | \int_{N_d(X)}\bigotimes_{j=1}^dg_jd\mu^{(d)}-
\int_{N_d(X)}\bigotimes_{j=1}^df_jd\mu^{(d)} \right |.
\end{split}
\end{equation}

Now by Pointwise Ergodic Theorem for $\Z^2$ applying to $(n,m)\mapsto T^{n+(j-1)m}$ (see for example \cite{Lin}) we have that for all $1\le j\le d$
\begin{equation}\label{need}
    \frac {1}{N^2} \sum_{n\in [0,N-1] \atop{m\in
[0,N-1]}}\Big | f_j(T^{n+(j-1)m}x)-g_j(T^{n+(j-1)m}x)\Big | \lra
\|f_j-g_j\|_1, \quad N\to \infty.
\end{equation}
for $\mu$ a.e.
Hence by Lemma \ref{lem-product},
\begin{equation}\label{yexd2}
\begin{split}
&  \limsup_{N\to \infty} \left | \frac{1}{N^2} \sum_{n\in [0,N-1] \atop{m\in [0,N-1]}}
   \prod_{j=1}^{d} f_j(T^{n+(j-1)m}x) -
\frac{1}{N^2} \sum_{n\in [0,N-1] \atop{m\in [0,N-1]}}
   \prod_{j=1}^d g_j(T^{n+(j-1)m}x)\right |\\
& \le \sum_{j=1}^d \Big[\lim_{N\to\infty }\frac {1}{N^2} \sum_{n\in [0,N-1]
\atop{m\in
[0,N-1]}}\Big | f_j(T^{n+(j-1)m}x)-g_j(T^{n+(j-1)m}x)\Big |\Big ]\\
& = \sum_{j=1}^d \|f_j-g_j\|_1\le \ep,\ a.e.
\end{split}
\end{equation}
Since $g_1\otimes \ldots \otimes g_d: X^{d}\rightarrow \R$ is
continuous and $(N_d(X), \langle \tau_d, \sigma_d \rangle, \mu^{(d)})$ is uniquely ergodic, we have
\begin{equation}\label{xdye3}
\lim_{N\to\infty} \left | \frac{1}{N^2} \sum_{n\in [0,N-1] \atop{m\in [0,N-1]}}\prod_{j=1}^d g_j(T^{n+(j-1)m}x)-
\int_{N_d(X)}\bigotimes_{j=1}^dg_jd\mu^{(d)} \right |=0.
\end{equation}
Because the $j^{th}$ marginal of $\mu^{(d)}$ is equal to $\mu$, by Lemma \ref{lem-product} we have
\begin{equation}\label{yexd4}
\left | \int_{N_d(X)}\bigotimes_{j=1}^dg_jd\mu^{(d)}-
\int_{N_d(X)}\bigotimes_{j=1}^df_jd\mu^{(d)} \right |
\le \sum_{j=1}^d\int_{X}|g_j-f_j|d\mu \le \ep.
\end{equation}

So combining (\ref{yexd1})-(\ref{yexd4}), we have
$$ \limsup_{N\to\infty }\left | \frac{1}{N^2} \sum_{n\in [0,N-1] \atop{m\in [0,N-1]}} \prod_{j=1}^{d} f_j(T^{n+(j-1)m}x) -
\int_{N_d(X)}\bigotimes_{j=1}^df_jd\mu^{(d)} \right |\le 2\ep, a.e. $$
Since $\ep$ is arbitrary, the proof is completed.

\subsection{Proof of Theorem D}

\begin{proof}

First we assume that $(X, T,\mu)$ is the strictly
ergodic system obtained in Theorem \ref{model} by setting $X=\h{X}$
with the unique measure $\mu^{(d+1)}$.

Let $\pi_{d-1}: X\lra Z_{d-1}$ be the factor map and $\mu=\int _{Z_{d-1}}\theta_s\ d\mu_{d-1}(s)$ be the disintegration
of $\mu$ over $\mu_{d-1}$. Then by Theorem \ref{thm4.5},
we have
\begin{equation}\label{s1}
\mu^{(d+1)}=\int_{Z_{d-1}}\theta_s\times \nu^{(d)}_s \ d\mu_{d-1}(s).
\end{equation}
For $x\in X$, let
\begin{equation}\label{s2}
    \mu^{(d)}_x=\nu^{(d)}_{\pi_{d-1}(x)}.
\end{equation}
By definition, for $\mu$ a.e. $x\in X$, $\mu^{(d)}_x$ is ergodic under $T\times T^2\times \ldots \times T^d$.

\medskip

Now we verify that $\{\mu^{(d)}_x\}_{x\in X}$ satisfies (\ref{ss}). First together with $\mu=\int _{Z_{d-1}}\theta_s\ d\mu_{d-1}(s)$,
we can rewrite (\ref{s1}) as
\begin{equation}\label{s3}
\mu^{(d+1)}=\int_{X}\d_x\times \mu^{(d)}_x \ d\mu(x).
\end{equation}

In fact,
\begin{equation*}
\begin{split}
\int_{X}\d_x\times \mu^{(d)}_x \ d\mu(x)=&\int_{Z_{d-1}} \int_X \d_x\times \nu^{(d)}_{\pi_{d-1}(x)} \ d \theta_s(x) \ d \mu_{d-1}(s) \\
   = & \int_{Z_{d-1}} \Big ( \int_X \d_x \ d \theta_s(x) \Big) \times \nu^{(d)}_{s} \ d \mu_{d-1}(s)
   \\ = &\int_{Z_{d-1}}\theta_s\times \nu^{(d)}_s \ d\mu_{d-1}(s)= \mu^{(d+1)}.
\end{split}
\end{equation*}

Now we show (\ref{ss}). By Theorem 1.1 in \cite{HK05}, let the left side of (\ref{ss}) converge (in $L^2$) to some function $g$. Now we show $g$ is equal to the right side of
(\ref{ss}). Let $f\in L^\infty(X)$, we have
\begin{equation*}
\begin{split}
& \int f(x)g(x)\ d \mu(x)\\ =&\lim_{N\to\infty}\int_X \frac{1}{N}
\sum_{n=0}^{N-1} f(x)
   f_1(T^nx)f_2(T^{2n}x)\ldots f_d(T^{d n}x)d\mu(x)\\
   = & \lim_{N\to \infty} \int_X\frac{1}{N}
\sum_{n0}^{N-1} f(T^nx)
   f_1(T^{2n}x)f_2(T^{3n}x)\ldots f_d(T^{(d+1)n}x)d\mu(x)
   \\ =& \int_{X^{d+1}} f(x_0)f_1(x_1)\ldots f_{d}(x_{d})d \mu^{(d+1)}(x_0,x_1,\ldots,x_{d-1}) \ (\text{by Definition \ref{de-Furstenberg-selfjoining}})
   \\ =& \int_X f(x)\Big(\int_{X^{d}} f_1(x_1)f_2(x_2)\ldots f_{d}(x_{d})\
   d \mu^{(d)}_x(x_1,x_2,\ldots,x_{d})\Big) d\mu(x).
\end{split}
\end{equation*}
Thus $g(x)=\displaystyle \int_{X^d} f_1(x_1)f_2(x_2)\ldots
f_{d}(x_{d})\    d \mu^{(d)}_x(x_1,x_2,\ldots,x_{d}), \ \mu \ a.e. \ x\in X$. Note that (\ref{s3}) is
used in the last equality.

\medskip

For $j\in \{1,2,\ldots,d\}$, $(p_j)_*(\mu_x^{(d)})$ is a $T^j$-invariant measure of $X$. Let $\nu_j=(p_j)_*(\mu_x^{(d)})$. Since
$(X,T,\mu)$ is uniquely ergodic, it is easy to see that
$\mu=[{\nu_j +T_*\nu_j +\ldots +(T^{j-1})_* \nu_j}]/{j}$. It follows that
$(p_j)_*(\mu_x^{(d)})=\nu_j\ll \mu$.
Hence the theorem holds for the system $(X,\mu, T)$.

\medskip

Now we prove the result for any ergodic system $(X,\X,\mu,T)$. By
the proof of Theorem A, $(X,\X,\mu,T)$ has a strictly ergodic model
$(\h{X},\h{T}, \h{\mu})$.
Let $\phi: X\rightarrow \h{X}$ be
the isomorphism. It is clear that $\phi^d: X^d\rightarrow \h{X}^d$
is also an isomorphism.

By the proof above, we have showed that for $(\h{X},\h{T}, \h{\mu})$, there exists a family
$\{\h{\mu}^{(d)}_x\}_{x\in \h{X}}$ of probability measures on ${\h{X}}^d$ such that it
satisfies condition (1)-(3) listed in the theorem. Define $\mu^{(d)}_x=\h{\mu}^{(d)}_{\phi(x)}\circ \phi^d$. By (3), for $\h{\mu}$ a.e. $\h{x}\in \h{X}$, $(p_j)_*(\h{\mu}_{\h{x}}^{(d)})\ll \h{\mu}$ for $1\le j\le d$, where $p_j: \h{X}^d\rightarrow \h{X}$ is the projection to the $j$-th coordinate. It follows that $\mu^{(d)}_x$ is well-defined.
Then it is not hard to check that $\{\mu^{(d)}_x\}_{x\in X}$ also satisfies (1)-(3). The proof is completed.
\end{proof}

\section{Proof of Theorem C}\label{section-thC}
In this section we will prove Theorem C. To do this, first we derive some properties from
the result proved in the previous sections. Then using the properties and a lemma we show that
the pointwise convergence can be lifted from a distal system to its isometric extension under some
conditions. Finally we conclude Theorem C by the structure theorem for distal systems.

\subsection{Isometric extensions}


Isometric extensions and weakly mixing extensions are two basic extensions in the Furstenberg structure theorem for a m.p.t.
Let $\pi: (X,\X,\mu,T)\rightarrow (Y,\Y,\nu,S)$ be a factor map. The
$L^2(X,\X,\mu)$ norm is denoted by $||\cdot||$ and the
$L^2(X,\X,\mu_y)$ norm by $||\cdot||_y$ for $\nu$-almost every $y
\in Y$. Recall $\{\mu_y\}_{y \in Y}$ is the disintegration of $\mu$
relative to $\nu$.
A function $f\in L^2(X,\X,\mu)$ is {\em almost periodic over $\Y$}
if for every $\ep>0$ there exist $g_1,\ldots, g_l \in
L^{2}(X,\X,\mu)$ such that for all $n\in \Z$
$$\min_{1\le j\le l} ||T^nf-g_j||_y<\ep$$
for $\nu$ almost every $y\in Y$. One writes  $f\in AP(\Y)$.
Let $K(X|Y, T)$ be the closed subspace of $L^2(X)$ spanned by the
almost periodic functions over $\Y$. When $\Y$ is trivial,
$K(X,T)=K(X|Y,T)$ is the closed subspace spanned by eigenfunctions
of $T$.

$X$ is an {\em isometric extensions} of $Y$ if $K(X|Y,T)=L^2(X)$ and
it is a {\em (relatively) weak mixing extension} of $Y$ if
$K(X|Y,T)=L^2(Y)$.

It can be shown that if $X$ is an isometric extension of an m.p.t. $(Y, \nu, S)$, then $X$ is isomorphic to a {\em skew product} $X'=Y\times M$,
where $M = G / H$ is a homogeneous compact metric space, $\mu'=\nu\times m_M$
with $m_M$ is the unique probability measure invariant under the transitive group
of isometries $G$. Moreover, the action of $T'$ on $X'$ is given by
$$T'(y,gH)=(Sy, \rho(y)gH),$$ where $\rho: Y\rightarrow G$ is a {\em cocycle}.
We denote $X'$ by $Y\times _\rho G/H$, and $T'$ by $T_\rho$. When $H$ is trivial,
we say $Y\times _\rho G$ is a {\em group extension} of $Y$. We refer to \cite{Glasner}
for the details.



\begin{lem}\label{extension}
Let $\pi:(X,\X,\mu,T)\rightarrow (Y,\Y,\nu,S)$ be a factor map between ergodic systems with $Z_{d-1}(X)=Z_{d-1}(Y)$, and $d\in \N$. Assume that $\{\mu^{(d)}_x\}_{x\in X}$
and $\{\nu^{(d)}_y\}_{y\in Y}$  are the families of measures defined in Theorem D respectively.
Then for given $f_1,\ldots, f_d\in L^\infty(\mu)$, one has that for $\mu$ a.e.  $x\in X$
\begin{equation}\label{xx}
\begin{split}
\int_{X^d} & f_1(x_1) \ldots f_d(x_d)\ d\mu^{(d)}_x(x_1,\ldots,x_d) \\
   = & \int_{Y^d} \E (f_1|\Y)(y_1)\E(f_2|\Y)(y_2)\ldots \E(f_d|\Y)(y_d)\
   d \nu^{(d)}_y(y_1,y_2,\ldots,y_{d}) .
\end{split}
\end{equation}
\end{lem}

\begin{proof}
Since $Z_{d-1}(X)=Z_{d-1}(Y)$, by Theorem 12.1 in \cite{HK05}, $Y$ is also a characteristic factor of $X$. That is
\begin{equation*}
\begin{split}
\|\frac{1}{N} \sum_{n=0}^{N-1}
   &f_1(T^nx)f_2(T^{2n}x)\ldots f_d(T^{dn}x)\\
   - & \frac{1}{N} \sum_{n=0}^{N-1}
   \E (f_1|\Y)(T^nx)\E(f_2|\Y)(T^{2n}x)\ldots \E(f_d|\Y)(T^{dn}x) \|_{L^2}\to 0
\end{split}
\end{equation*}
as $N\to \infty$. Moreover we obtain \eqref{xx} by applying Theorem D (2) to $(X,\X,\mu,T)$ and $(Y, \Y,\nu, S)$.
\end{proof}

In the proof of Proposition \ref{IsoExten}, we need the following Proposition \ref{Les-d}.  Recall that for a compact metric space $X$, $M(X)$ is the set of all Borel
probability measure on $X$ with weak$^*$ topology. If $T$ is a continuous map from $X$ to itself, then it is well known that for all $x\in X$ each limit point of $\{\frac 1N \sum_{n=0}^{N-1} T^n_*\d_x\}_{N\in \N}$ is $T$-invariant. If $T$ is measurable instead of continuity, then more will be involved. Proposition \ref{Les-d} will deal with the similar situation in $X^d$ for our purpose.

We remark that when $d=2$ and all transformations
are ergodic, this proposition was proved in \cite[Proposition 3]{L87}.
We leave the proof of this result in the appendix, which is similar to the one in \cite{L87} but much more involved.

\begin{prop}\label{Les-d}
Let $(X,\X,\mu)$ be a probability space with $X$ compact metric space and $d\in \N$. For $1\le i\le d$,
let $T_i: X\rightarrow X$ be measure preserving transformations with finitely many ergodic components.
Then there is a measurable set $X_*$ with $\mu(X_*)=1$ such that for $x\in X_*$
each weak$^*$ limit point $\lambda$ of the sequence
$$\big\{\frac {1}{N} \sum_{n=0}^{N-1} (T_1\times T_2\times \ldots \times T_d)^n_* \d_{(x,x,\ldots,x)}
\big\}_N$$ is in $M(X^d)$, and $\lambda$ is $T_1\times \ldots \times T_d$-invariant.
\end{prop}

The following proposition is crucial for our proof.
\begin{prop}\label{IsoExten}
Let $\pi:(X=Y\times _\rho G/H,\X,\mu,T)\rightarrow (Y,\Y,\nu,S)$ be an isometric extension
between two ergodic systems with $Z_{d-1}(X)=Z_{d-1}(Y)$, and $d\in \N$. If
\begin{equation*}
    \frac 1 N\sum_{n=0}^{N-1}f'_1(T^ny)f'_2(T^{2n}y)\ldots f'_d(T^{dn}y)
\end{equation*}
converge $\nu$ a.e. for any given $f'_1,\ldots, f'_d\in
L^\infty(\nu)$, then
\begin{equation*}
    \frac 1 N\sum_{n=0}^{N-1}f_1(T^nx)f_2(T^{2n}x)\ldots f_d(T^{dn}x)
\end{equation*}
converge $\mu$ a.e. for any given $f_1,\ldots, f_d\in
L^\infty(\mu)$.
\end{prop}

\begin{proof}
We may assume that $Y$ is a compact metric space.
By the assumption of the theorem, there is some measurable set $Y_0\in \Y$ with $\nu(Y_0)=1$ such that for $y\in Y_0$
and for all $f'_1,\ldots, f'_d\in C(Y)$
\begin{equation*}
    \frac 1 N\sum_{n=0}^{N-1}f'_1(T^ny)f'_2(T^{2n}y)\ldots f'_d(T^{dn}y)
\end{equation*}
converge since $C(Y)$ is separable. By Theorem D, we may assume that for all $y\in Y_0$ and for all $f'_1,\ldots, f'_d\in C(Y)$,
\begin{equation*}
    \frac 1 N\sum_{n=0}^{N-1}f'_1(T^ny)f'_2(T^{2n}y)\ldots f'_d(T^{dn}y)\longrightarrow \int_{Y^d} f'_1(y_1)f'_2(y_2)\ldots f'_{d}(y_{d})\
   d \nu^{(d)}_y(y_1,y_2,\ldots,y_{d})
\end{equation*}
as $N \to \infty$ since the almost everywhere limit coincides with the limit in $L^2$.

Let $X_0=\{(y,gH): y\in Y_0, g\in G\}$. Then $\mu(X_0)=\mu(Y_0\times G/H)=1$.
Since $C(X)$ is a separable space, by Lemma \ref{extension}, there is measurable set $X_1$ such that
$\mu(X_1)=1$ and (\ref{xx}) holds for all continuous functions. In Proposition \ref{Les-d}, we take $T_1=T,\ldots,T_d=T^d$
and let $X_*$ be the set defined there. 

\medskip

Now fix $x=(y,gH)\in X_0\cap X_1\cap X_*$. Let $\lambda$ be a weak$^*$ limit point of the sequence
$$\big\{\frac {1}{N} \sum_{n=0}^{N-1} (T\times T^2\times \ldots \times T^d)^n_* \d_{(x,x,\ldots,x)}\big\}_N$$ in $M(X^d)$.
By Proposition \ref{Les-d}, $\lambda$ is $T\times \ldots \times T^d$-invariant. We are going to show that $\lambda=\mu^{(d)}_x$.

\medskip


Let $\pi^{d}: X^d\rightarrow Y^{d}, (x_1,\ldots, x_d)\mapsto (\pi(x_1),\ldots,\pi(x_d))$.
Then $\pi^{d}_*\lambda$ is a weak$^*$ limit point of the sequence
$$\big\{\frac {1}{N} \sum_{n=0}^{N-1} (S\times S^2\times \ldots \times S^d)^n_* \d_{(y,y,\ldots,y)}\big\}_N$$ in $M(Y^d)$.
By the assumption, we know
$$\lim_{N\to \infty}\frac {1}{N} \sum_{n=0}^{N-1} (S\times S^2\times \ldots \times S^d)^n_* \d_{(y,y,\ldots,y)}=\nu^{(d)}_y.$$
Thus $\pi^{d}_*\lambda=\nu^{(d)}_y$.

Let $\psi_1,\ldots, \psi_d\in C(G)$ such that $\psi_i\ge 0$ and $\int_G \psi_i d m=1$
for all $i\in \{1,\ldots, d\}$.
Now define a new measure $\lambda_{\psi_1,\ldots, \psi_d}$ as follows:
\begin{equation}\label{}
\begin{split}
& \lambda_{\psi_1,\ldots, \psi_d}(f_1\otimes \ldots \otimes f_d)  \\
= & \int_{G^d\times X^d} f_1(y_1, h_1g_1H)\ldots f_d(y_d, h_d g_dH)\psi_1(h_1)\ldots
\psi_d(h_d)\ d h_1\ldots d h_d d \lambda(x_1,\ldots,x_d),
\end{split}
\end{equation}
where $x_i=(y_i,g_iH)$ and $f_i\in C(X)$ for all $i\in \{1,2, \ldots, d\}$.

Then by $\pi^{(d)}_*\lambda=\nu^{(d)}_y$, (\ref{xx}) and $\E(|f_i||\Y)(y_i)=\int_G |f_j(y_j, h_iH)| d h_i$, $1\le j\le d$  we have
\begin{equation*}
\begin{split}
& |\lambda_{\psi_1,\ldots, \psi_d}(f_1\otimes \ldots \otimes f_d) | \\
 \le & \prod_{i=1}^d \sup_{g\in G} |\psi_i(g)| \int_{G^d\times X^d} |f_1(y_1, h_1g_1H)
 \ldots f_d(y_d, h_d g_dH)| d h_1\ldots d h_d d \lambda(x_1,\ldots,x_d)\\
  = & \prod_{i=1}^d \sup_{g\in G} |\psi_i(g)| \int_{G^d\times Y^d} |f_1(y_1, h_1H)
  \ldots f_d(y_d, h_d H)| d h_1\ldots d h_d d \nu^{(d)}_y (y_1,\ldots,y_d) \\
  =&\ \prod_{i=1}^d \sup_{g\in G} |\psi_i(g)|  \int_{Y^d} \E (|f_1||\Y)(y_1)\E(|f_2||\Y)(y_2)\ldots \E(|f_d||\Y)(y_d)\
   d \nu^{(d)}_y(y_1,y_2,\ldots,y_{d}) \\
  =& \prod_{i=1}^d \sup_{g\in G} |\psi_i(g)| \int_{X^d}
  |f_1(x_1)\ldots f_d(x_d)| d \mu^{(d)}_x (x_1,\ldots,x_d).
\end{split}
\end{equation*}
Thus we have
\begin{equation}\label{}
    \lambda_{\psi_1,\ldots, \psi_d} \ll \mu^{(d)}_x.
\end{equation}

Note $\rho: Y\rightarrow G$ is a cocycle. For each $n\in \N$, let
$$\rho^{(n)}(y)=\rho(S^{n-1}y)\rho(S^{n-2}y)\ldots \rho(y).$$
Then we have $$T^n(y,gH)=(S^n, \rho^{(n)}(y)gH).$$

Now in addition we assume that $\{\psi_i\}$ satisfy $\psi_i(h^{-1}gh)=\psi_i(g)$ for all $1\le i\le d$.
Then for $f_1,\ldots,f_d\in C(X)$, we have
\begin{equation*}
\begin{split}
& \lambda_{\psi_1,\ldots, \psi_d}((f_1\otimes \ldots \otimes f_d)\circ T\times \ldots \times T^d) \\
  = & \int_{G^d\times X^d} f_1(Sy_1, \rho(y_1)h_1g_1H)\ldots f_d(S^dy_d,
  \rho^{(d)}(y_d)h_d g_dH) \psi_1(h_1)\ldots \psi_d(h_d) d h_1\ldots d h_d d \lambda(x_1,\ldots,x_d)\\
  = & \int_{G^d\times X^d} f_1(Sy_1, h_1\rho(y_1)g_1H)\ldots f_d(S^dy_d,
  h_d\rho^{(d)}(y_d) g_dH) \psi_1(h_1)\ldots \psi_d(h_d) d h_1\ldots d h_d d
  \lambda(x_1,\ldots,x_d)\\
  = & \int_{G^d\times X^d} f_1(y_1, h_1g_1H)\ldots f_d(y_d, h_d g_dH)
  \psi_1(h_1)\ldots \psi_d(h_d)\ d h_1\ldots d h_d d \lambda(x_1,\ldots,x_d)\\
  = & \lambda_{\psi_1,\ldots, \psi_d}(f_1\otimes \ldots \otimes f_d) .
\end{split}
\end{equation*}
That is, $ \lambda_{\psi_1,\ldots, \psi_d}$ is $T\times \ldots \times T^d$-invariant. Since $\mu^{(d)}_x$
is ergodic, we have that
$$ \lambda_{\psi_1,\ldots, \psi_d}=\mu^{(d)}_x.$$

Now we will define a sequence $\{\phi_n\}_n$ such that all $\phi_n$ satisfies
the properties which $\psi_i$ hold above, and $$\lambda_{\phi_n,\ldots, \phi_n}\to \lambda,\ n\to \infty .$$
Then we get that $\lambda=\mu^{(d)}_x$.

Since $G$ is a compact metric group, there is an invariant metric $\varrho$. For all $n\in \N$, let
$$\varphi_n(g)= 1/n -\inf \{ 1/n, \varrho(e, g)\}.$$
Then let
$$\phi_n=\frac{\varphi_n}{\int_G\varphi_n d m}.$$
Note that $\phi_n$ is supported on $A_n=\{g\in G: \varrho(e,g)<\frac{1}{n}\}$. It follows that
for given $y_1,\ldots, y_d$, $f_1(y_1, h_1g_1H)\ldots f_d(y_d, h_d g_dH)$ is close to $f_1(y_1, g_1H)\ldots f_d(y_d, g_dH)$
uniformly on $A_n$. Then using the fact that $\int_G \phi_n d m=1$  we deduce that $\phi_n$ is what we need.

\medskip

To sum up, we have proved that for all $x\in X_0\cap X_1\cap X_*$, $\mu^{(d)}_x$
is the unique weak$^*$ limit point of sequence
$$\big\{\frac {1}{N} \sum_{n=0}^{N-1} (T\times T^2\times \ldots \times T^d)^n_* \d_{(x,x,\ldots,x)}\big\}_N$$ in $M(X^d)$.
Hence for all $x\in X_0\cap X_1\cap X_*$, we have that for all $f_1,\ldots, f_d\in
C(X)$
\begin{equation}\label{s4}
\begin{split}
\frac{1}{N} \sum_{n=0}^{N-1}
   &f_1(T^nx)f_2(T^{2n}x)\ldots f_d(T^{dn}x)\\
   \longrightarrow & \int_{X^d} f_1(x_1)f_2(x_2)\ldots f_{d}(x_{d})\
   d \mu^{(d)}_x(x_1,x_2,\ldots,x_{d})
\end{split}
\end{equation}
as $N\to \infty$. Note that $\mu(X_0\cap X_1\cap X_*)=1$.

\medskip

Now by the same approximation argument as in the proof of Theorem B, we have that for all $f_1,\ldots, f_d\in L^\infty(\mu)$,
\begin{equation*}
    \frac 1 N\sum_{n=0}^{N-1}f_1(T^nx)f_2(T^{2n}x)\ldots f_d(T^{dn}x)
\end{equation*}
converges $\mu$ a.e.. The proof is completed.
\end{proof}

\subsection{Proof of Theorem C}
In this final subsection we will prove Theorem C. We start with the definition of a distal system.

\begin{de}
Let $\pi: (X,X,\mu,T)\rightarrow (Y,\Y,\nu,T)$ be a factor between two ergodic systems.
We call the extension $\pi$ a {\em distal extension} if there exists a countable ordinal $\eta$
and a directed family of factors $(X_\theta, \X_\theta, \mu_\theta, T), \theta \le \eta$ such that
\begin{enumerate}
  \item $X_0=Y$ and $X_\eta=X$.
  \item For $\theta<\eta$ the extension $\pi_\theta : X_{\theta+1}\rightarrow X_\theta$ is isometric and non-trivial (i.e. not an isomorphism).
  \item For a limit ordinal $\lambda \le \eta$, $X_\lambda={\lim\limits_{\longleftarrow}}_{\theta<\lambda} X_\theta$ (i.e. $\X_\lambda =\bigvee X_\theta$).
\end{enumerate}

If $X$ is a distal extension of the trivial system, then $(X,\X,\mu,T)$ is called a {\em distal system}.
\end{de}

\medskip

\noindent {\bf The proof of Theorem C:} We say a system $(X,\X,\mu, T)$ satisfies ($\divideontimes$), if for all $f_1, \ldots, f_d \in L^{\infty}(\mu)$
\begin{equation*}
    \frac 1 N\sum_{n=0}^{N-1}f_1(T^nx)\ldots f_d(T^{dn}x)
\end{equation*}
converge $\mu$ a.e.. The aim is to prove each distal system satisfies ($\divideontimes$). We will use the structure of distal systems and Proposition \ref{IsoExten} to complete the proof.

\medskip

Let $\pi_{d-1}: X\rightarrow Z_{d-1}$ be the factor map. Then $\pi_{d-1}$
is distal since $X$ is distal. By the definition of a distal extension,
there exists a countable ordinal $\eta$ and a directed family of factors
$(X_\theta, \X_\theta, \mu_\theta, T), \theta \le \eta$ such that
\begin{enumerate}
  \item $X_0=Z_{d-1}$ and $X_\eta=X$.
  \item For $\theta<\eta$ the extension $\pi_\theta : X_{\theta+1}\rightarrow X_\theta$ is non-trivial isometric.
  \item For a limit ordinal $\lambda \le \eta$, $X_\lambda={\lim\limits_{\longleftarrow}}_{\theta<\lambda} X_\theta$.
\end{enumerate}
\medskip

Then we have:

($i$) By Theorem \ref{ziegler}, $X_0=Z_{d-1}$ satisfies ($\divideontimes$).

\medskip

($ii$) For $\theta<\eta$ the extension $\pi_\theta : X_{\theta+1}\rightarrow X_\theta$ is
non-trivial isometric. If $X_\theta$ satisfies ($\divideontimes$),
then by Proposition \ref{IsoExten}, $X_{\theta+1}$ satisfies ($\divideontimes$).

\medskip

($iii$) For a limit ordinal $\lambda \le \eta$, if for all $\theta<\lambda$,
$X_\theta$ satisfies ($\divideontimes$), then it is easy to verify that the inverse
limit $X_\lambda={\lim\limits_{\longleftarrow}}_{\theta<\lambda} X_\theta$ also satisfies ($\divideontimes$).

\medskip
($iv$) By ($i$-$iii$), $X_\eta=X$ satisfies ($\divideontimes$).

The proof is completed.




\bigskip

\appendix

\section{Proof of Proposition \ref{Les-d}}


\begin{proof}[Proof of Proposition \ref{Les-d}]
Notice that when all transformations $T_i$ are continuous, the result follows by the standard argument and in this case $X_*=X$.

Now we deal with the general case. Since $X$ is a compact metric space, $C(X)$ is separable.
Let $C_1$ be a countable dense subset of $C(X)$. Let
$$C_2=\{|g_1\circ T_i-g_2|^d: 1\le i \le d,\ g_1,g_2\in C_1 \}.$$
It is obvious that $C_2$ is countable. Since $T_i$ has only finitely many ergodic components, it follows
that $\I_{T_i}$ is a finite $\sigma$-algebra and hence is generated by a finite measurable partition $\beta_i$
for each $1\le i\le d$. For a given $f\in C_2$, there
is some $X_f\in \X$ with $\mu(X_f)=1$ such that for all $x\in X_f$,  one has $\mu(\beta_i(x))>0$ and
\begin{equation}\label{Les1}
    \lim_{N\to \infty}\frac{1}{N}\sum_{n=0}^{N-1} f(T_i^nx)=\E(f| \I_{T_i})(x)
    =\frac{\int_{\beta_i(x)} f(y) d\mu(y)}{\mu(\beta_i(x))},\ \  \forall \ 1\le i\le d,
\end{equation}
where $\beta_i(x)$ is the atom of $\beta_i$ containing $x$.
Let
$$X_*=\bigcap_{f\in C_2} X_f .$$
Since $C_2$ is countable, we conclude that $X_*\in \X$ and $\mu(X_*)=1$.

Recall that the topology of $C(X)$ is the uniform convergence topology. Let
$$ C_3=\{|g_1\circ T_i-g_2|^d: 1\le i \le d,\ g_1,g_2\in C(X) \}.$$
Then each element of $C_3$ is the uniform limit of elements of $C_2$. It is easy to
show that $(\ref{Les1})$ holds for all $x\in X_*$ and for all $f\in C_3$.

\medskip

Let $x_0\in X_*$ and let $\lambda$ be a weak limit point of the sequence
$$\big\{\frac {1}{N} \sum_{n=0}^{N-1} (T_1\times T_2\times \ldots \times T_d)^n_* \d_{(x_0,x_0,\ldots,x_0)}\big\}_N.$$
Now we show $\lambda$ is $T_1\times \ldots \times T_d$-invariant.

For all $f_1,\ldots,f_d\in C(X)$, we have that
$$\Big|\frac{1}{N}\sum_{n=0}^{N-1}\prod_{i=1}^df_i(T^n_ix_0)\Big |\le \prod_{i=1}^d\Big(\frac{1}{N}
\sum_{n=0}^{N-1}|f_i|^d(T^n_ix_0)\Big)^{1/d}.$$
Since for each $1\le i\le d$, $|f_i|^d\in C_3$, by $(\ref{Les1})$ we have
\begin{equation*}
    \limsup_{N\to \infty} \Big|\frac{1}{N}\sum_{n=0}^{N-1}\prod_{i=1}^df_i(T^n_ix_0)\Big |\le
    \prod_{i=1}^d\Big(\E(|f_i|^d|\I_{T_i})(x_0)\Big)^{1/d}.
\end{equation*}
In particular, we deduce that
\begin{equation}\label{Les2}
    \lambda(f_1\otimes\ldots\otimes f_d)\le \prod_{i=1}^d\Big(\E(|f_i|^d|\I_{T_i})(x_0)\Big)^{1/d}.
\end{equation}

Now we show that $\lambda(f_1\circ T_1\otimes \ldots \otimes f_d\circ T_d)=\lambda(f_1\otimes \ldots \otimes f_d)$
for all $f_1,\ldots,f_d\in C(X)$. Let $M>0$ such that $\|f_i\|_\infty\le M, \ 1\le i\le d$.
Let $\delta=\min_{1\le i\le d}\{ \mu(\beta_i(x_0))\}$. Then $\delta>0$, and hence for any $\ep>0$, one can choose
functions $g_i\in C(X), \ 1\le i\le d$ such that
\begin{equation*}
    \|f_i\circ T_i-g_i \|_{L^d(\mu)}<\ep \delta^{1/d} \text{ and }\ \|g_i\|_\infty\le M, \ \ 1\le i\le d.
\end{equation*}
Thus
\begin{equation}\label{Les3}
    \Big(\E(|f_i\circ T_i-g_i |^d\big|\I_{T_i})(x_0)\Big)^{1/d}=\Big(\frac{\int_{\beta_i(x_0)}
     |f_i\circ T_i-g_i |^d(y) d\mu(y)}{\mu(\beta_i(x_0))}\Big)^{1/d}
    <\ep
\end{equation}
for $1\le i\le d$. By Lemma \ref{lem-product} and $(\ref{Les2})$,
\begin{equation}\label{Les4}
\begin{split}
& |\lambda(f_1\circ T_1\otimes \ldots \otimes f_d\circ T_d)-\lambda(g_1\otimes \ldots \otimes g_d)| \\
\le &\sum_{i=1}^d \big| \lambda(\bigotimes_{j=1}^{i-1} f_j\circ T_j \otimes (f_i\circ T_i-g_i)\otimes\bigotimes_{k=i+1}^d g_k)\big|\\
\le & \sum_{i=1}^d M^{d-1}\Big(\E(|f_i\circ T_i-g_i|^d\big|\I_{T_i})(x_0)\Big)^{1/d}\le dM^{d-1}\ep.
\end{split}
\end{equation}
Also we have
\begin{equation}\label{Les5}
\begin{split}
& \Big|\frac{1}{N}\sum_{n=0}^{N-1}\prod_{i=1}^df_i(T^n_ix_0) -\frac{1}{N}\sum_{n=0}^{N-1}\prod_{i=1}^dg_i(T^n_ix_0)\Big| \\
\le & \frac{1}{N}\Big|\prod_{i=1}^df_i(x_0)-\prod_{i=1}^df_i(T_i^Nx_0)\Big|+\Big|\frac{1}{N}\sum_{n=0}^{N-1}\Big[
\prod_{i=1}^df_i(T^{n+1}_ix_0) -\prod_{i=1}^d g_i(T^n_ix_0)\Big]\Big|\\
\le &\frac{2M^d}{N}+\sum_{i=1}^d \Big| \frac{1}{N}\sum_{n=0}^{N-1}\big(\prod_{j=1}^{i-1}f_j(T^{n+1}_jx_0)\big)(f_i(T_i^{n+1}x_0)-g_i(T^n_ix_0))\big(
\prod_{k=i+1}^dg_k(T^n_kx_0)\big)\Big|\\
\le & \frac{2M^d}{N}+\sum_{i=1}^d M^{d-1} \Big(\frac{1}{N}\sum_{n=0}^{N-1} \big|f_i\circ T_i-g_i\big|^d(T^n_ix_0)\Big)^{1/d}.
\end{split}
\end{equation}
By $(\ref{Les1})$ and $(\ref{Les3})$, it follows that
\begin{equation}\label{Les6}
\begin{split}
& \limsup_{N\to \infty} \Big|\frac{1}{N}\sum_{n=0}^{N-1}\prod_{i=1}^df_i(T^n_ix_0) -\frac{1}{N}\sum_{n=0}^{N-1}\prod_{i=1}^dg_i(T^n_ix_0)\Big| \\
\le & \sum_{i=1}^d M^{d-1} \Big(\E\big(|f_i\circ T_i-g_i\big|^d\big|\I_{T_i}\big)(x_0)\Big)^{1/d}\le dM^{d-1}\ep .
\end{split}
\end{equation}
In particular,
\begin{equation}\label{Les7}
    \Big| \lambda(f_1\otimes \ldots \otimes f_d)-\lambda(g_1\otimes \ldots \otimes g_d) \Big|\le dM^{d-1}\ep.
\end{equation}
By $(\ref{Les4})$ and $(\ref{Les7})$, we conclude that
\begin{equation}\label{Les8}
    \Big|\lambda(f_1\circ T_1\otimes \ldots \otimes f_d\circ T_d)- \lambda(f_1\otimes \ldots \otimes f_d) \Big|\le 2dM^{d-1}\ep.
\end{equation}
Since $\ep$ is arbitrary, the above inequality implies that
$$\lambda(f_1\circ T_1\otimes \ldots \otimes f_d\circ T_d)=\lambda(f_1\otimes \ldots \otimes f_d)$$
for all $f_1,\ldots, f_d\in C(X)$. Since the linear span of $C(X)^d$ is dense in $C(X^d)$, we have
$$\lambda (F\circ (T_1\times \ldots \times T_d))=\lambda (F)$$
for all $F\in C(X^d)$. That is, $\lambda$ is $T_1\times \ldots \times T_d$-invariant.
\end{proof}


\end{document}